\documentclass[reqno]{amsart}     
\usepackage{amssymb, amsmath, amsthm, epsf, epsfig}
\usepackage[a4paper,vmargin={3cm,3cm},hmargin={3cm,3cm}]{geometry}

\usepackage{bbm}
\usepackage{xcolor}
\usepackage{enumerate}
\usepackage{comment}

\renewcommand{\(}{\left(}
\renewcommand{\)}{\right)}
\renewcommand{\d}{\mathrm{d}}
\newtheorem{theo}{Theorem}
\newtheorem{prop}{Proposition}
\newtheorem{lemma}{Lemma}
\newtheorem{cor}{Corollary\!\!}

\newtheorem{ncor}{Corollary}
 
\theoremstyle{definition}

\newtheorem{df}{Definition}
\newtheorem{ex}{Example\!\!}

\theoremstyle{remark}

\newtheorem{rem}{Remark\!\!}
 
\newtheorem{nrem}{Remark} 
\newtheorem{ack}{Acknowledgment\!\!}

\newcommand{\beq}{\begin{equation}} 
\newcommand{\eeq}{\end{equation}} 
\newcommand{\bal}{\begin{align}} 
\newcommand{\eal}{\end{align}} 
\newcommand{\bals}{\begin{align*}} 
\newcommand{\eals}{\end{align*}} 
\newcommand{\barr}[1]{\begin{array}{#1}} 
\newcommand{\earr}{\end{array}}

\newcommand{\bth}{\begin{theo}} 
\newcommand{\bl}{\begin{lemma}} 
\newcommand{\el}{\end{lemma}} 
\newcommand{\bp}{\begin{prop}} 
\newcommand{\ep}{\end{prop}} 
\newcommand{\bdf}{\begin{df}} 
\newcommand{\edf}{\end{df}} 
\newcommand{\brem}{\begin{rem}} 
\newcommand{\erem}{\end{rem}} 
\newcommand{\bnrem}{\begin{nrem}} 
\newcommand{\enrem}{\end{nrem}} 
\newcommand{\bex}{\begin{ex}} 
\newcommand{\eex}{\end{ex}} 
\newcommand{\bcor}{\begin{cor}} 
\newcommand{\ecor}{\end{cor}} 
\newcommand{\bncor}{\begin{ncor}} 
\newcommand{\encor}{\end{ncor}} 
\newcommand{\bpf}{\begin{proof}} 
\newcommand{\epf}{\end{proof}}

\def\({\left(} 
\def\){\right)}

\numberwithin{equation}{section}

\title{Effective Erd\H{o}s--Wintner Theorems for Digital Expansions}

\author{Michael Drmota$^*$, and Johann Verwee$^*$}

\thanks{{}$^*$ TU Wien, Institute of Discrete Mathematics and Geometry,
Wiedner Hauptstrasse 8-10, A-1040 Vienna, Austria. michael.drmota@tuwien.ac.at. johann.verwee@tuwien.ac.at. Research 
supported by the Austrian  Science Foundation FWF, project F55-02.}

\begin{document}

\begin{abstract}
In 1972 Delange \cite{D} observed in analogy of the classical Erd\H os-Wintner theorem 
that $q$-additive functions $f(n)$ has a distribution function if and only if the two series 
$\sum f(d q^j)$, $\sum f(d q^j)^2$ converge. The purpose of this paper is to provide
quantitative versions of this theorem as well as generalizations to other kinds of 
digital expansions. In addition to the $q$-ary and Cantor case we focus on the 
Zeckendorf expansion that is based on
the Fibonacci sequence, where we provide a sufficient and necessary condition
for the existence of a distribution function, namely that the two series 
$\sum f(F_j)$, $\sum f(F_j)^2$ converge (previously only a sufficient condition
was known \cite{BG1}).
\end{abstract}

\maketitle

\section{Introduction}
\label{sec1}

The classical Erd\H os-Wintner theorem \cite{ErdosI,ErdosII,ErdosIII,EW} states that a real valued 
additive function $f$, that is defined by the property $f(mn) = f(m) + f(n)$
for coprime positive integers $m,n$, has a distribution function
\[
F(y) = \lim_{N\to\infty} \frac 1N \# \{ n < N : f(n) \leqslant  y \}
\]	
if and only if the following three series converge:
\[
\sum_{|f(p)|> 1} \frac 1p, \quad \sum_{|f(p)| \leqslant 1} \frac{f(p)}p, \quad \sum_{|f(p)| \leqslant 1} \frac{f(p)^2}p.
\]
This is a proper analogue of Kolmogorov's three series theorem in probability theory.
Recently Tenenbaum and the second author have studied effective versions of this theorem \cite{TV}.

The purpose of the present paper is to consider the analogue problem in the context of 
digital expansions, where we restrict ourselves to the (common) $q$-ary digital expansion,
to Cantor digital expansions, and to the Zeckendorf expansion.
Both, the prime decomposition as well as digital expansions are expansions of integers.
Whereas the prime decomposition encodes the multiplicative structure, digital expansions
encode -- in some way -- the additive structure of the integers. 

Note that a function is additive if and only if
\[
f(p_1^{e_1} \cdots p_r^{e_r}) = f(p_1^{e_1}) + \cdots + f(p_r^{e_r}),
\]
where $p_1,\ldots, p_r$ are different primes and $e_1,\ldots, e_r$ positive integers.

In a similar way one defines so-called $q$-additive functions by
\[
f( d_1 q^{e_1} + \cdots + d_r q^{e_r}) = f(d_1 q^{e_1}) + \cdots + f(d_r q^{e_r}),
\]
where $q\geqslant 2$ is a given integer, $e_1< e_2 < \cdots < e_r$ are different positive integers and
$d_1,\ldots, d_r$ are integers digits satisfying $1\leqslant d_j \leqslant q-1$ ($1\leqslant j\leqslant r$).

For example, the $q$-ary sum-of-digits function $s_q(n)$, defined by $s_q(d q^a) = d$, 
is $q$-additive as well as the $q$-ary Van-der-Corput sequence
$v_q(n)$, defined by $v_q(d q^a) = dq^{-a-1}$. 

The analysis of $q$-additive functions and their distribution (as well as their multiplicative counterpart) have
attained a lot of attention during the last few decades, see for example
\cite{Ber,zbMATH05015294,zbMATH05014334,Drm2,zbMATH05879513,DrmGut,zbMATH07080247,Kat2,Mad,zbMATH06471888,zbMATH02203397,zbMATH05836546,zbMATH05712759,zbMATH00981690}.

As already indicated, in the present paper we want to focus on analogues of the Erd\H os-Wintner theorems 
for digital expansions and their quantitative versions. 
For the $q$-adic case there is already a proper analogue by Delange \cite{D}
saying that a real-valued $q$-additive function 
$f(n)$ has a distribution function $F(y)$
if and only if the two series 
\[
\sum_{j\geqslant 0}\sum_{d=1}^{q-1}f(dq^j)\quad\text{and}\quad
\sum_{j\geqslant 0}\sum_{d=1}^{q-1}f(dq^j)^2
\]
converge. 
By applying a Berry-Esseen inequality we present a quantified version of
this theorem and give examples, the most interesting one is related to Cantor-Lebesgue measures (Section~\ref{sec:qary}).

The results for $q$-ary expansions can be easily extended to Cantor digital
expansions with bounded quotients (Section~\ref{sec:Cantor}).

Finally we discuss the Zeckendorf digital expansions, where the base sequence
are the Fibonacci numbers (Section~\ref{sec:Zeckendorf}). 
This is actually the most challenging case. First of all we give a full
characterization for the existence of a distribution function (so far only
a sufficient condition was known \cite{BG1}). This requires a delicate analysis
of Fibonacci-like recurrences with non-constant coefficients.
A quantitative version can be then established under more general hypotheses. 

\section{$q$-Ary Digital Expansions}\label{sec:qary}

We start with a self-contained proof of the Theorem~\ref{ch8:ThEWDel} by Delange \cite{D}
since we will use similar principles later. We then give a quantitative version (Section~\ref{sec:qarysub2})
and after an example we discuss Cantor-Lebesgue measures (Section~\ref{sec:qarysub3}).

\subsection{The $q$-ary Erd\H os-Wintner theorem}\label{sec:qarysub1}

\begin{theo}[Delange \cite{D}]\label{ch8:ThEWDel}
Let $f(n)$ be a real-valued $q$-additive function. Then 
$f(n)$ has a distribution function $F(y)$, that is 
\begin{equation}\label{ch8:eqDel}
\lim_{N\to\infty}\frac 1N \#\left\{n<N\mid f(n) \leqslant  y\right\}=F(y),
\end{equation}
if and only if the two series 
\begin{equation}\label{ch8:eqDel2}
\sum_{j\geqslant 0}\sum_{d=1}^{q-1}f(dq^j)\quad\text{and}\quad
\sum_{j\geqslant 0}\sum_{d=1}^{q-1}f(dq^j)^2
\end{equation}
converge. In this case the characteristic function $\varphi(t)$ of the limiting distribution
is given by
\begin{equation}\label{eqvarphi}
\varphi(t) = \int_{-\infty}^\infty e^{ity} \; \d F(y) =
\prod_{j\geqslant 0} \left( \frac 1q  \sum_{d=0}^{q-1} e^{it f(d q^j)}  \right).
\end{equation}
\end{theo}

Before (re-)proving Theorem~\ref{ch8:ThEWDel} we mention that 
a theorem of Jessen and Wintner asserts that such a distribution measure given by $F(y)$ is pure (that is, it is either  absolutely continuous having a density or purely singular continuous, where $F$  is continuous and has zero derivative almost everywhere, or consists only of point masses).
Note that the last alternative can only occur
if and only if there exists an integer $J > 0$ such that $f(dq^j) = 0$ for $j \geqslant J$ and all $d\in \{1,\ldots q-1\}$: in this case the distribution consists only of finitely many point masses, see  \cite[Lemma 1.22]{E}.

\begin{proof}
By L\'{e}vy's theorem it is sufficient to show that the characteristic functions converge for every
fixed $t\in \mathbb{R}$ :
\begin{equation}\label{eqvarphiconv}
\lim_{N\to\infty} \frac 1N \sum_{n<N} e^{it f(n)} = \varphi(t),
\end{equation}
where the limit function $\varphi(t)$ is continuous at $t=0$. 
Since 
\[
\varphi_{q^L}(t) := \frac 1{q^L} \sum_{n<q^L} e^{it f(n)} = \prod_{j=0}^{L-1} \left( \frac 1q  \sum_{d=0}^{q-1} e^{it f(d q^j)}     \right) 
\]
it is clear that the limit has to be
\[
\varphi(t) = \prod_{j\geqslant 0} \left( \frac 1q  \sum_{d=0}^{q-1} e^{it f(d q^j)}     \right).
\]
(Note that we are using the convention $f(0) = 0$.)

Set
\begin{align}
m_{j,q}:= \frac 1q \sum_{d = 1}^{q-1} f(dq^j), \quad
m_{2;j,q}^2:=\frac 1q\sum_{d = 1}^{q-1} f(dq^j)^2.  \label{eqmjdef}
\end{align}
Then by using the relation $e^{iu} = 1 + iu + O(u^2)$ for real $u$ we have
\begin{align*}
\log\left( \frac 1q \sum_{d=0}^{q-1} e^{it f(d q^j)} \right) &= 
\log \left( 1 + it m_{j,q} + O\left( t^2 m_{2;j,q}^2 \right) \right) \\
&= it m_{j,q} + O\left( t^2 (m_{j,q}^2 + m_{2;j,q}^2) \right) \\
&= it m_{j,q} + O\left( t^2 m_{2;j,q}^2 \right) 
\end{align*}
provided that $j$ is sufficiently large (depending on $t$).
Note thay by Cauchy-Schwarz's inequality $|m_{j,q}| \le \sqrt{m_{2;j,q}^2}$.
Hence if the two series \eqref{ch8:eqDel2} converge the limit $\lim_{L\to\infty} \varphi_{q^L}(t)$ exists and, thus, equals
$\varphi(t)$.
It is an easy exercise (by using the $q$-ary expansion of $N$, see also Lemma~\ref{Ledifference}) 
that the convergence $\varphi_{q^L}(t) \to \varphi(t)$
implies the convergence (\ref{eqvarphiconv}), too. Finally it is not difficult to check that 
$\varphi(t)$ is continuous at $t=0$ (later we will even show that $\varphi(t) = 1 + O(t)$ as $t\to 0$).

Conversely, suppose that we have  $\lim_{L\to\infty} \varphi_{q^L}(t) = \varphi(t)$, where $\varphi(t)$ is continuous 
at $t=0$. Clearly we have, too,
$\lim_{L\to\infty} |\varphi_{q^L}(t)| = |\varphi(t)|$, and we can choose  $t_0> 0$ such that $|\varphi(t)| \geqslant \frac 12$
for all $t\in [0,t_0]$. We now use the fact that $1-\cos(x) = 2\sin^2(x/2) \geqslant 8 \|x/(2\pi) \|^2$ (where $\|u\|$ denotes the distance from $u$
to the set of integers) and obtain
\begin{align*}
\left| \sum_{d=0}^{q-1} e^{it f(d q^j)} \right|^2  &= 
\sum_{d_1,d_2 =0}^{q-1} \cos\( t \( f(d_1 q ^j) - f(d_2 q^j)\)\) \\
&\leqslant  q^2 - 2 \sum_{d=0}^{q-1} \( 1 - \cos\(t f(d q^j) \)\) \\
&\leqslant  q^2 - 16 \sum_{d=0}^{q-1} \left\| \frac{t f(d q^j)}{2\pi} \right\|^2
\end{align*} 
and 
\[
\left| \frac 1q \sum_{d=0}^{q-1} e^{it f(d q^j)} \right| \leqslant \exp\left( - \frac 8{q^2} \sum_{d=0}^{q-1} \left\| \frac{t f(d q^j)}{2\pi} \right\|^2 \right).
\]
Now, if $0< t \leqslant  t_0$ we obtain
\begin{equation}\label{eqphiest}
\frac 12 \leqslant |\varphi(t)| = \prod_{j\geqslant 0} \left| \frac 1q \sum_{d=0}^{q-1} e^{it f(d q^j)} \right|
\leqslant \exp\left( - \frac 8{q^2} \sum_{j\geqslant 0} \sum_{d=0}^{q-1} \left\| \frac{t f(d q^j)}{2\pi} \right\|^2 \right)
\end{equation}
and, thus, 
\begin{equation}\label{eqsumssquares}
\sum_{j\geqslant 0} \sum_{d=0}^{q-1} \left\| \frac{t f(d q^j)}{2\pi} \right\|^2 \leqslant  c_1
\end{equation}
for some constant $c_1> 0$. In particular it follows that the sequence 
$\| t f(d q^j) /(2\pi) \|^2$ ($1\leqslant  d \leqslant  q-1$, $j\geqslant 0$) converges to $0$ for all $t\in (0,t_0]$.
Actually this also implies that $f(d q^j)$ converges to $0$.\footnote{For the reader's convenience
we append a proof of this property in the Appendix.} In particular the sequence $f(d q^j)$ is bounded.
Thus, we can choose $t\in (0,t_0]$ such that $|t f(d q^j)| \leqslant \pi$ for all $d$ and $j$.
By (\ref{eqsumssquares}) this implies that the sum
\[
\sum_{j\geqslant 0} \sum_{d=0}^{q-1} \left( \frac{t f(d q^j)}{2\pi} \right)^2 
\]
is bounded and consequently convergent for $t\in (0,t_0]$. Since $t> 0$ we also get convergence of the 
sum $\sum_{j\geqslant 0} \sum_{d=0}^{q-1} f(d q^j)^2$. 

Now the final step is easy. By using the fact that $\lim_{L\to\infty} \varphi_{q^L}(t) = \varphi(t)$
and by taking logarihms (as above) it follows that the sum
\[
\sum_{j\geqslant 0} m_{j,q} = \frac 1q \sum_{j\geqslant 0}  \sum_{d = 1}^{q-1} f(dq^j)
\]
converges, too.
\end{proof}

\subsection{An Effective Version of the $q$-ary Erd\H os-Wintner Theorem}\label{sec:qarysub2}

Clearly, every $q$-additive function is determined by the values $f(d q^j)$, 
where $d \leqslant q-1$ and $j\geqslant 0$. There is one very simple case that we briefly discuss first.
Suppose that only finitely many values $f(d q^j)$ are non-zero, that is, there exists
$J > 0$ such that $f(d q^j) = 0$ for $j \leqslant J$ and all $d\leqslant q-1$. Then $f(n)$ is 
periodic with period $q^J$ and by (\ref{eqvarphi}) the limiting distribution function $F$
equals $F_{q^J}$. Furthermore, if we write $N = q^J Q + r$ for some $0\leqslant r < q^J$ it follows 
that 
\[
\#\{ n < N : f(n) \leqslant y\} = Q\,\#\{ n < q^J : f(n) \leqslant y\} + \#\{ n < r : f(n) \leqslant y\} 
\]
which implies 
\[
F_N(y) = \frac 1N\#\{ n < N : f(n) \leqslant y\} = \frac{Q\,q^J}{N} F(y) + O\left( \frac {q^J}N \right) 
= F(y) + O\left( \frac {q^J}N \right).
\]
In different terms
\[
\|F_N- F\|_\infty  \ll \frac 1N
\]
which is the optimal convergence rate (the implicit constant depends of $f$, of course).

From now on we will assume that $f(dq^j) \ne 0$ for infinitely many instances. 

\begin{theo}\label{Th2A}
Let $f(n)$ be a real-valued $q$-additive function such that the two series {\normalfont (\ref{ch8:eqDel2})} converge and
that $f(dq^j) \ne 0$ for infinitely many instances. 
Set  $L = \lfloor \log_q N \rfloor$.
Then we have for all real numbers $T\geqslant 1$ such that $h = \lceil \log_q( T \log T ) \rceil \leqslant L$
\begin{equation}\label{eqTh2-ext}
\|F_N - F\|_{\infty} \ll Q_F\left(\frac{1}{T}\right)
+ T \left| \sum_{j > L} \sum_{d=0}^{q-1} f(dq^j) \right| +
T \sqrt{h} \sqrt{ \sum_{j\geqslant L-h} \sum_{d=0}^{q-1} f(dq^j)^2 },
\end{equation}
where $F_N$ denotes the distribution function of $\(f(n): n<N\)$,
$F$ the limiting distribution function, and 
\begin{equation}\label{eqQfdef}
Q_F(h) = \sup_{z\in \mathbb{R}} \, (F(z+h)-F(z)) \qquad (h>0).
\end{equation}
The implicit constant depends on $q$.
\end{theo}

Note that the upper bound (\ref{eqTh2-ext}) refers explicitly 
to the two series {\normalfont (\ref{ch8:eqDel2})} and goes to $0$ with $N\to\infty$
	if these two series converge and if $Q_F(h)\to 0$ as $h\to 0$ (and with $T$ going slowly enough to $+\infty$). 
Actually, we can improve the bound (\ref{eqTh2-ext}) by 
\begin{align}
\|F_N - F\|_{\infty} &\ll Q_F\left(\frac{1}{T}\right) + T  \sum_{j=L-h+1}^L \sum_{d=0}^{q-1} |f(dq^j)|   \label{eqTh2-ext-2} \\
&+ T \left| \sum_{j > L} \sum_{d=0}^{q-1} f(dq^j) \right| +
T^2 \sum_{j > L} \sum_{d=0}^{q-1} f(dq^j)^2
 \nonumber
\end{align}
that is slightly better but also more involved. (Clearly by applying 
Cauchy-Schwarz's inequality and the bound $|u| \ll u^2$ for bounded $u$ (\ref{eqTh2-ext-2})
implies (\ref{eqTh2-ext}).)

It seems there is no simple and tight upper bound for $Q_F(h)$ in terms of $f$.
Clearly, if  $F$ is absolutely continuous and $F'$ is bounded, then $Q_F(h) \ll h$
(which is by the way optimal since we always have $Q_F(h) \gg h$ for $0\le h < 1$).
If $\varphi \in L^{p}(\mathbb{R})$, then we can use the following upper bound (see \cite[lemme III.2.9]{Tenen})
\begin{equation}\label{eqQfineg}
Q_F(h) \ll h \int_{-1/h}^{1/h} |\varphi(t)|\, \d t.
\end{equation}
and H\"{o}lder's inequality to obtain $Q_F(h) \ll h^{1/p}$.

In general we can only use proper estimates for $|\varphi(t)|$ that are deduced from the
product representation (\ref{eqvarphi}) and  (\ref{eqQfineg}). 
In particular this leads to (\ref{eqLe1}) (see Lemma~\ref{Le1}).
By applying this bound and by doing also some {\it modifications} in order to 
simplify it for applications we formulate the following variant. This variant applies in particular
if the series
\begin{equation}\label{eqabsvalconv}
\sum_{j\geqslant 0} \frac 1q \sum_{d=0}^{q-1} |f(dq^j)|
\end{equation}
converges and does not necessarily give an upper bound that tends to zero if the two series
(\ref{ch8:eqDel2}) converge.

\begin{theo}\label{Th2B}
Let $f(n)$ be a real-valued $q$-additive function such that the series {\normalfont (\ref{eqabsvalconv})} converge and
that $f(dq^j) \ne 0$ for infinitely many instances. Set
\[
S(t) = \{(d,j) \in \{1,\ldots,q-1\}\times \mathbb{N} : |f(dq^j)| \leqslant \pi/ |t| \} \, (t>0).
\]
Then with $c_1= 2/(\pi^2 q^2)$ and the abbreviations $L = \lfloor \log_q N \rfloor$ and $h = \lceil \log_q( T \log T ) \rceil$, we have for all real numbers $T\geqslant 1$ such that $h \leqslant L$
\begin{align}
\|F_N - F\|_{\infty} &\ll \frac 1T \int_{0}^T  \exp\left( - c_1 \, t^2 \sum_{(d,j)\in S(t)} f(d q^j)^2 \right)\, \d t 
+ T  \sum_{j=L-h+1}^L \sum_{d=0}^{q-1} |f(dq^j)|   \label{eqTh2} \\
&+ \int_{1/T}^T \min\left\{ \frac 1{1+t}, \sum_{j > L} \sum_{d=0}^{q-1} |f(dq^j)| \right\} 
\exp\left( - c_1 \, t^2 \sum_{(d,j)\in S(t),\, j \leqslant  L} f(d q^j)^2 \right) \, \d t  \nonumber
\end{align}
where $F_N$ denotes the distribution function of $\(f(n): n<N\)$ and $F$ the limiting distribution function $\normalfont{(}$the implicit constant depending on $\left.q\)$.
\end{theo}

We note that the sums 
\begin{equation}\label{eqeasysum}
t^2 \sum_{(d,j)\in S(t)} f(d q^j)^2
\end{equation}
could be replaced by the better bound
\begin{equation}\label{eqnoteasysum}
\frac 8{q^2} \sum_{j\geqslant 0} \sum_{d=0}^{q-1} \left\| \frac{t f(d q^j)}{2\pi} \right\|^2,
\end{equation}
see the proof of Lemma~\ref{Le1} below (clearly if $(d,j)\in S(t)$ then 
$\| t f(d q^j) /(2\pi) \|^2 = t^2 f(d q^j)^2 /(4\pi^2)$, and the other terms are just neglected).
However, the sum (\ref{eqeasysum}) is easier to handle than the sum (\ref{eqnoteasysum}). Nevertheless,
there are cases, where we need the better bound (\ref{eqnoteasysum}), see Section~\ref{sec:qarysub3}.

Our analysis is based on the Berry-Esseen inequality \cite[Lemma 1.47]{E} that we state first:
\begin{prop}
Let $F$ and $G$ be two distributions functions of probability distributions with
characteristic functions $\varphi$ and $\psi$, respectively. Then we have for all $T>0$
\[
\|F-G\|_\infty \ll Q_F\left( \frac 1T \right) + \int_{-T}^T \left| \frac{\varphi(t)- \psi(t)}{t} \right|\, \d t.
\]
\end{prop}

In order to apply the Berry-Esseen inequality we need some information on the characteristic 
functions $\varphi(t)$ and $\varphi_N(t)$. The proofs of Theorems~\ref{Th2A} and \ref{Th2B} will
be done in parallel. We first give an upper bound of $\varphi(t)$ in terms of $f$.
\begin{lemma}\label{Le1}
Let $\varphi(t)$ be the characteristic function given by {\normalfont(\ref{eqvarphi})}.
Then {\normalfont (with the notation of Theorem {\normalfont \ref{Th2A}})} we have
\begin{equation}\label{eqLe1}
|\varphi(t)| \leqslant \exp\left( - \frac 2{\pi^2 q^2} t^2 \sum_{(d,j)\in S(t)} f(d q^j)^2 \right).
\end{equation}
\end{lemma}

\begin{proof}
From (\ref{eqphiest}) we directly obtain
\begin{align*}
|\varphi(t)| &\leqslant \exp\left( - \frac{8}{q^2}\,\sum_{j\geqslant 0} \sum_{d=0} \left\| \frac{t f(d q^j)}{2\pi} \right\|^2 \right) \\
&\leqslant \exp\left( - \frac 2{\pi^2 q^2} t^2 \sum_{(d,j)\in S(t)} f(d q^j)^2 \right).
\end{align*}
\end{proof}

Next we consider the difference $\varphi(t)-\varphi_N(t)$ for $|t|\leqslant 1/T$.
For this purpose we recall the definitions of $m_{j,q}$ and $m_{2;j,q}$ from (\ref{eqmjdef}).
It is easy to obtain the following two representations:
\[
\frac 1{q^L} \sum_{n< q^L} f(n) = \sum_{\ell < L} m_{\ell,q}, \quad
\frac 1{q^L} \sum_{n< q^L} f(n)^2 =  \sum_{\ell < L} (m_{2;\ell,q}^2-m_{\ell,q}^2)  + \left( \sum_{\ell < L} m_{\ell,q} \right)^2.
\]
From this we easily obtain the following property.
\begin{lemma}\label{LeO1}
Suppose that $f(n)$ is a real-valued $q$-additive function such that the two series {\normalfont(\ref{ch8:eqDel2})} converge.
Then we have
\[
\frac 1N \sum_{n< N} |f(n)| = O(1)
\]
as $N\to \infty$;  the implicit constant depends on $f$.
\end{lemma}

\begin{proof}
Suppose that $L = \lceil \log_q N \rceil$. Then $N \leqslant q^L < q N$ and we  have
\begin{align*}
\frac 1N \sum_{n<N } f(n)^2 &\leqslant \frac q{q^L} \sum_{n<q^L} f(n)^2 \\
&= q \sum_{\ell < L} (m_{2;\ell,q}^2-m_{\ell,q}^2) + q\left( \sum_{\ell < L} m_{\ell,q} \right)^2 \\
&= O(1).
\end{align*}
Recall that  $m_{j,q}^2  \leqslant m_{2;j,q}^2$. Consequently by Cauchy-Schwarz's inequality
\[
\frac 1N \sum_{n< N} |f(n)| \leqslant \left( \frac 1N \sum_{n< N} f(n)^2 \right)^{1/2} = O(1)
\]
which completes the proof of the lemma.
\end{proof}

By definition and the elementary inequality $|e^{it}-1| \leqslant |t|$ we have
\begin{align*}
|1 - \varphi_N(t)| &= \left| \frac 1N \sum_{n<N} (1-e^{itf(n)}) \right| \\
&\leqslant |t|\,  \frac 1N \sum_{n<N} |f(n)| \\
&= O(|t|).
\end{align*}
By taking the limit $N\to \infty$ we also obtain $|1-\varphi(t)| = O(t)$.
This also implies that $\varphi_N(t) - \varphi(t) = O(t)$ and consequently
the ratio $\frac 1t(\varphi_N(t) - \varphi(t))$ stays bounded.  
In particular, the following integral can be trivially bounded:
\[
\int_{-1/T}^{1/T} \left| \frac{\varphi_N(t)- \varphi(t)}{t} \right|\,  \d t \ll \frac 1T.
\]
Note that $Q_F(1/T) \gg 1/T$. Thus we can neglect the term $1/T$.

Finally we have to deal with $\varphi(t)-\varphi_N(t)$ for $1/T <|t|\leqslant T$.
Since we have no direct access to $\varphi_N$ we have to approximate it by $\varphi_{q^L}$,
where $L = \lfloor \log_q N \rfloor$ and  where we can apply the following lemma by 
Delange \cite[Proposition 3]{D} (properly adjusted).

\begin{lemma}\label{Ledifference}
Suppose that $f(n)$ is real-valued and $q$-additive. Then we have for all integers $h\geqslant 1$ and $N \geqslant  q^h$
\begin{equation}\label{eqLedifference}
|\varphi_N(t) - \varphi_{q^{L+1}}(t) | \leqslant \frac 2{q^{h-1}} + 2\sqrt 2  \sum_{j=L-h+1}^L  \max_{1\le d \le q-1} (1-\cos(t f(dq^j))^{1/2}
\end{equation}
where $L = \lfloor \log_q N \rfloor$. 
\end{lemma}

We note that \cite{D} states the slightly worse bound
\begin{equation}\label{eqLedifference2}
|\varphi_N(t) - \varphi_{q^{L+1}}(t) | \leqslant \frac 2{q^{h-1}} + 2 \sqrt{2h} 
\left( \sum_{j=L-h+1}^L  \max_{1\leqslant d\leqslant q-1} \(1-\cos(t f(dq^j))\) \right)^{1/2}
\end{equation}
The final step in \cite{D} is to apply the Cauchy-Schwarz inequality so that 
(\ref{eqLedifference}) implies (\ref{eqLedifference2}).

\medskip

Since $1-\cos(x) = 2 \sin^2(x/2) \leqslant x^2/2$ we, thus, obtain
\[
\int\limits_{1/T < |t| \leqslant  T} \frac{|\varphi_N(t) - \varphi_{q^{L+1}}(t) |}{|t|} \, \d t \ll
\frac{\log T}{q^h} + T   \sum_{j=L-h+1}^L \, \sum_{d=0}^{q-1} |f(dq^j)|.
\]
The choice $h = \lceil \log_q( T \log T ) \rceil$ ensures that  $(\log T) /q^h \leqslant 1/T$.

What remains is to consider the corresponding integral over the difference $|\varphi_{q^{L+1}}(t) - \varphi(t)|$.
Obviously we have
\[
\varphi_{q^{L+1}}(t) - \varphi(t) = \varphi_{q^{L+1}}(t) \left( 1- \frac{\varphi(t)}{\varphi_{q^{L+1}}(t) }   \right),
\]
where
\[
\frac{\varphi(t)}{\varphi_{q^{L+1}}(t) } = \prod_{j >  L} \left( \frac 1q  \sum_{d=0}^{q-1} e^{it f(d q^j)}  \right)
\]

We can either use the trivial estimate $|\varphi_{q^{L+1}}(t) |\le 1$ or proceed as in Lemma~\ref{Le1} to get
\[
|\varphi_{q^{L+1}}(t) | \leqslant \exp\left( - \frac 2{\pi^2 q^2} t^2 \sum_{(d,j)\in S(t),\ j \leqslant L} f(d q^j)^2 \right).
\]
Furthermore, we trivially have $|1- \varphi(t)/\varphi_{q^{L+1}}(t) | \leqslant 2$ and (by the properties $\varphi(t) = 1 + O(t)$ and
$\varphi_{q^{L+1}}(t)  = 1 + O(t)$ as $t\to 0$) $|1- \varphi(t)/\varphi_{q^{L+1}}(t) | = O(t)$. This implies
\[
\frac{|1- \varphi(t)/\varphi_{q^{L+1}}(t) |}{|t|} \ll \frac 1{1+|t|}.
\]

By using the properties  $e^{iu} = 1 + iu + O(u^2)$, $|e^{iu}-1| \le |u|$ and $1 + |u| \le e^{|u|}$ (for real $u$) we have 
\[
\prod_{j >  L}  \left(\frac 1q  \sum_{d=0}^{q-1} e^{it f(d q^j)}\right) = 
\exp\left( it \sum_{j > L} \frac 1q  \sum_{d=0}^{q-1} f(dq^j) + 
O\left( t^2 \sum_{j > L} \frac 1q  \sum_{d=0}^{q-1} f(dq^j)^2 \right) \right)
\]
and, thus, 
\begin{align*}
\left| \prod_{j >  L}  \left(\frac 1q  \sum_{d=0}^{q-1} e^{it f(d q^j)}\right)  - 1 \right| &\leqslant
\left| \exp\left( c t^2\sum_{j > L} \frac 1q  \sum_{d=0}^{q-1} f(dq^j)^2 \right) - 1 \right| \\
&+ \left| t  \sum_{j > L} \frac 1q  \sum_{d=0}^{q-1} f(dq^j) \right|.
\end{align*}
Since the left hand side is upper bounded by $2$ this also implies
\begin{align*}
\left| \prod_{j >  L}  \left(\frac 1q  \sum_{d=0}^{q-1} e^{it f(d q^j)}\right)  - 1 \right| &\leqslant
c' t^2\sum_{j > L} \frac 1q  \sum_{d=0}^{q-1} f(dq^j)^2  \\
&+ \left| t  \sum_{j > L} \frac 1q  \sum_{d=0}^{q-1} f(dq^j) \right|
\end{align*}
for some constant $c'> 0$.
This directly provides the upper bound
\[
\int\limits_{1/T < |t| \leqslant  T} \frac{|\varphi_{q^{L+1}}(t)  - \varphi(t)|}{|t|} \, \d t \ll
T \left|  \sum_{j > L} \frac 1q  \sum_{d=0}^{q-1} f(dq^j) \right|
+ T^2 \sum_{j > L} \frac 1q  \sum_{d=0}^{q-1} f(dq^j)^2
\]
and completes the proof of (\ref{eqTh2-ext-2}) and consequently that of Theorem~\ref{Th2A}.

Finally, by using the inequality
\[
\left| \prod_{k=1}^K a_k - 1 \right| \leqslant \sum_{k=1}^K |a_k-1|
\]
for complex numbers $a_k$ with $|a_k| \leqslant 1$, we have
\begin{align*}
\left| \prod_{j >  L}  \left(\frac 1q  \sum_{d=0}^{q-1} e^{it f(d q^j)}\right)  - 1 \right| &\leqslant
\sum_{j >  L} \left|   \frac 1q  \sum_{d=0}^{q-1} e^{it f(d q^j)} - 1 \right| \\
&\leqslant \frac {|t|} q  \sum_{j >  L} \sum_{d=0}^{q-1} |f(dq^j)| 
\end{align*}
provided that the last sum converges. Clearly this estimate completes the proof of Theorem~\ref{Th2B}.

\begin{ex}
Suppose that $q = 2$ and that 
\[
c_1 j^{-\alpha} \leqslant f(2^j) \leqslant c_2 j^{-\alpha}
\]
for $j\geqslant  1$ and some $\alpha$ that satisfies $1 < \alpha < 2$, where $c_1,c_2$ are positive constants. 
For this kind of asymptotic behavior we obtain 
\[
\| F - F_N\|_\infty \ll (\log N)^{1-\alpha}
\]
as we will see in a moment.

First we have
\[
\sum_{ j \in S(t)} f(2^j)^2 \geqslant  \sum_{j > c_3 |t|^{1/\alpha}} j^{-2\alpha} \geqslant  c_4 |t|^{-2+1/\alpha}
\]
for proper positive constants $c_3,c_4$. Consequently 
\[
\int_{-\infty}^\infty \exp\left( - c \, t^2  \sum_{ j \in S(t)} f(2^j)^2 \right)\, \d t = O(1).
\]
Thus, the first term in Theorem~\ref{Th2B} is bounded by $O(1/T)$.

Next we set $T =  L^{\alpha-1}$, where $L = \lfloor \log_2 N  \rfloor$.
and $h = \lceil \log_2(T \log T) \rceil = (\alpha-1)\log_2 L + O(\log\log L)$. In particular this implies that
\[
\frac{\log T}{2^h} \ll \frac{\log L}{L^\alpha} \ll L^{1-\alpha}
\]
and (since $\alpha < 2$)
\[
 T \, \sum_{j=L-h+1}^L  |f(2^j)|
\ll  L^{\alpha -1} (\log L) \, L^{-\alpha} \ll \frac{\log L}{L} \ll L^{1-\alpha}.
\]
Finally we have
\[
\sum_{j > L} |f(2^j)| \ll L^{1-\alpha},
\]
and for $|t|\geqslant  1/T = L^{1-\alpha}$ we have 
\[
\sum_{ j \in S(t), j \leqslant L} f(2^j)^2 \geqslant  \sum_{c_3 |t|^{1/\alpha}< j \leqslant L} j^{-2\alpha} 
\geqslant  c_4 \left( |t|^{-2+1/\alpha} - L^{-2\alpha+1} \right) \geqslant  c_5 |t|^{-2+1/\alpha}.
\]
which implies that also the last term in Theorem~\ref{Th2B} is bounded by $L^{1-\alpha}$.

If $\alpha \geqslant  2$ we can do a similar analysis and obtain $\| F - F_N\|_\infty \ll \sqrt{\log\log N} \, (\log N)^{-\alpha/2}$.
\end{ex}

\subsection{Cantor-Lebesgue Measures}\label{sec:qarysub3}

We now discuss the binary case $q=2$ and the 2-additive function $f(n)$ that is given by
\[
f(2^j) = \beta^j \qquad (j \geqslant  0)
\]
for some $\beta \in (0,1)$. In all cases the limiting distribution $F = F_\beta$ is continuous, however,
it can be quite {\it wild} in general. 
Let $\mu_\beta$ denote the corresponding limiting measure. It is easy to see that $\mu_\beta$ is linked to the distribution of $\sum_{n=0}^{\infty} \pm \beta^n$ (where the signs are chosen independently with probability $1/2$), denoted by $\nu_{\beta}$, by the relation
\[
 \mu_\beta(B) = \nu_{\beta}\left(2B-\frac{1}{1-\beta}\right)
\]
for any real Borel set $B\in \mathcal{B}(\mathbb{R})$. 
We just have to compare the corresponding characteristic functions.

The following quote is from \cite{PSS}:
\begin{quote}
Kershner and Wintner (1935) observed that $\nu_\beta$
 is singular for $\beta \in (0,1/2)$ since it is supported on a
Cantor set of zero Lebesgue measure (in fact, $\nu_\beta$
 is the standard Cantor-Lebesgue measure on this
Cantor set). Wintner (1935) noted that $\nu_\beta$
 is uniform on $[-2, 2]$ for 
 $\beta = 1/2$ and for $\beta  = (2 -1/k)^{-1}$ with $k \geqslant  2$ it is absolutely continuous, 
with a density in $C^{k-2}(\mathbb{R})$. For $\beta \in (1/2,1)$ the support of $\nu_\beta$
is the interval $[-(1-\beta)^{-1},(1-\beta)^{-1}]$,  so one might surmise that $\nu_\beta$
 is absolutely continuous for all such $\beta$. However, in \cite{Er}  Erd\H os (1939) showed that 
 is singular when $\beta$  is the reciprocal of a Pisot
number (recall that a Pisot number is an algebraic integer all of whose conjugates are less than one
in modulus). This gives a closed countable set of $\beta \in (1/2,1)$ with $\nu_\beta$
singular.
\end{quote}

In the case $f(2^j) = \beta^j$ we cannot apply Theorem~\ref{Th2B} directly since the (easy) upper bound for $|\varphi(t)|$
is not sufficient to obtain any bound that tends to zero. Nevertheless, a more careful analysis provides
the following bound.

\begin{theo}\label{Th3}
Let $f(n)$ be the $2$-additive function defined by $f(2^j) = \beta^j$, $j\geqslant  0$, where $\beta \in (0,1)$. 
Then we have 
\[
\|F - F_N\| \ll_{\beta}  N^{- c(\beta) }
(\log N)^{  \log(1/\beta)/\log 2}
\]
for some exponent $c(\beta)>0$\footnote
{G\'erald Tenenbaum mentioned to us that a different approach to the concentration function
leads to the estimation $\| F - F_N \| \ll N^{- \overline{c}(\beta) }
(\log N)^{  \log(1/\beta)/\log 2}$ with a relatively simple explicit number $\overline{c}(\beta)$. 
We will discuss this in the Appendix B.}
(and the implicit constant depends on $\beta$, too).
Moreover, if $F$ is absolutely continuous, then we have
\[
\|F - F_N\| \ll_{\beta} N^{-\log(1/\beta)/\left(\log\,2+\log(2/\beta)\right)}\,(\log N)^{  \log(1/\beta)/\log 2}.
\]
\end{theo}
We note that these bounds are certainly not optimal, since $c = c(\beta)$ that we get from the 
following proof is usually a very small number. 
On the other hand, for $\beta = \frac 12$ the
resulting sequence is the Van-der-Corput sequence that has discrepancy of (optimal) order $(\log N)/N$
which is much better than what we obtain in Theorem~\ref{Th3}. In general we expect that optimal bounds 
are of the form $N^{-\tilde c(\beta) + o(1)}$ for some $\tilde c(\beta)>0$.

\medskip

In the present case the characteristic function is given by 
\[
\varphi(t) = \prod_{j\geqslant  0} \frac{ 1 + e^{it \beta^j} }2 
\]
and consequently
\[
|\varphi(t)| = \prod_{j\geqslant  0} \(\frac{ 1 + \cos(t \beta^j) }2\)^{1/2} \leqslant 
\exp\left( - \frac 12 \sum_{j\geqslant  0} \| t \beta^j /\pi \|^2 \right).
\]
It is essential to include the terms  $\| t \beta^j /\pi \|$ for which $|t \beta^j| > \pi/2$
although they behave quite erratic. However, in the average they contribute essentially and
this we will use.

Set $B = 1/\beta > 1$ and suppose that $T > 1$. Let $J_0\geqslant  0$ be the integer defined by $B^{J_0} \leqslant T < B^{J_0+1}$
Then we have
\begin{align*}
\int_1^T |\varphi(t)|\, \d t &\leqslant \sum_{J=0}^{J_0} \int_{B^J}^{B^{J+1}} \exp\left( - \frac 12 \sum_{j\geqslant  0} \| t \beta^j /\pi \|^2 \right) \, \d t  \\
&= \sum_{J=0}^{J_0} B^J \int_1^B  \exp\left( - \frac 12 \sum_{0\leqslant j\leqslant J} \| \tau B^{J-j}  /\pi \|^2 \right) \, \d\tau
\end{align*}
The most important ingredient for the proof of Theorem~\ref{Th3} is the following property.

\begin{lemma}\label{Leimportant}
Suppose that $B> 1$ and set {\normalfont (}for $J\geqslant  0${\normalfont )}
\[
S_J(\tau) = \frac 12 \sum_{0\leqslant j\leqslant J} \| \tau B^{j} /\pi \|^2.
\]
Then there exists $\eta> 0$ (depending on $\beta$) such that:
\begin{equation}\label{eqLeimportant}
\int_1^B \exp(-S_J(\tau) )\, \d\tau \ll \exp( - \eta J ).
\end{equation}
\end{lemma}

\begin{proof}
The idea of the proof is that the average value of $\frac 12 \| x \|^2$ is $1/24$. Thus, we can expect that
for most $\tau$ we have $S_J(\tau) \geqslant \eta J$ for some $\eta> 0$. This would lead to (\ref{eqLeimportant}).
However, it is not that easy to make this heuristic argument rigorous. For this purpose we adopt methods from
\cite{Berkes}.

We first assume $B\geqslant  2$ (the proof is very similar for $1<B<2$, even slightly easier).
We set
\[
f(x) = \frac{1}{24} - \frac{\left\|x\right\|^2}{2} = \sum_{k=1}^{\infty} c_k\,\cos(2\pi k x) \;\; (x \in \mathbb{R}),
\]
where
\[
c_k = \frac{(-1)^{k+1}}{2\pi^2 k^2}.
\]

\vspace*{0.5cm}

Furthermore, let $H \geqslant  1$ and $\gamma \geqslant  2$ be two integers such that
\begin{eqnarray*}
B^{4H} -2H^\gamma\,B^{3H} - 2\pi B^{2H} + 2\pi \geqslant  0.
\end{eqnarray*}
We also put for $x \in \mathbb{R}$ and $m \in \mathbb{N}$
\[
g(x) = \sum_{k=1}^{H^\gamma} c_k\,\cos(2\pi k x) \;\;\; \textrm{and} \;\;\; U_m(x) = \sum_{j=Hm+1}^{H(m+1)} g\left( (B-1)B^j x/\pi\right).
\]
For any $\lambda>0$, $\psi\in\{f,g\}$,  and $M\in \mathbb{N}$ we define
\[
\overline{S_{M,\psi}}(t) =\sum_{j=0}^{M} \psi\left( \frac{(B-1)B^{j}}{\pi}\,t+\frac{B^j}{\pi} \right)
\]
and 
\begin{eqnarray*}
	\chi_\psi(\lambda,M)= (B-1)\, \displaystyle\int_{0}^{1} \exp\left(\lambda\,\overline{S_{M,\psi}}(t)\right)\,\textrm{d}t.
\end{eqnarray*}
We have the relation
$$ \int_{1}^{B}  \exp\left(- S_{J}(\tau) \right) \, {\textrm{d}\tau} = \exp\left(-\frac{J+1}{24}\right)\,\chi_f(1,J).$$

In order to estimate $\chi_f(1,J)$ we introduce a parameter $\lambda >0$ and consider $\chi_f(\lambda,J)$. 
Clearly we have the following inequality for any $\kappa >0$,
$$ \chi_f(1,J) \leqslant  (B-1)\,\exp\(\frac{J+1}{24}\)\,\Lambda\(\left\{t \in [0,1]\,:\,\left|\overline{S_{J,f}}(t)\right| \geqslant  \kappa \right\}\) + (B-1)\,\mathrm{e}^{\kappa},$$
where $\Lambda$ denotes the Lebesgue measure.
By Cernov's bound we have all $\lambda >0$ 
$$ \Lambda\(\left\{t \in [0,1]\,:\,\left|\overline{S_{J,f}}(t)\right| \geqslant  \kappa \right\}\) \leqslant 
\frac 1{B-1} \mathrm{e}^{-\lambda \kappa} \chi_f(\lambda,J),$$ 
which implies
\begin{eqnarray}\label{inter1}
\int_{1}^{B}  \exp\left(- S_{J}(\tau) \right) \, {\textrm{d}\tau} \leqslant \mathrm{e}^{-\lambda \kappa} \chi_f(\lambda,J) 
+ (B-1)\,\exp\(\kappa-\frac{J+1}{24}\).
\end{eqnarray}

Since $\left|f(x)-g(x)\right| \leqslant  \pi^{-2}\,H^{-\gamma}$ we can estimate $\chi_f(\lambda,J)$ with the help of 
 $\chi_g(\lambda,J)$: 
\begin{eqnarray}\label{inter2}
\chi_f(\lambda,J) \leqslant    \exp\(\frac{\lambda\,J}{\pi^2\,H^{\gamma}}\)\,\chi_g(\lambda,J).
\end{eqnarray}

It remains to bound $\chi_g$. For simplicity we assume that $J = pH$, where $p$ is an even integer.
Thus, from the representation
\begin{eqnarray*}
\overline{S_{J,g}}(t) = \sum_{m=0}^{p/2} U_{2m}\(t+\frac{1}{B-1}\) + \sum_{m=1}^{p/2} U_{2m-1}\(t+\frac{1}{B-1}\)
\end{eqnarray*}
and an application of the Cauchy-Schwarz inequality we obtain
\begin{equation*}
\chi_g(\lambda,Hp) \leqslant  \left( \int_{0}^{1}  \exp\left(\lambda\,\sum_{m=0}^{p/2} U_{2m}(t+\xi)\right) \, \textrm{d}t \right)^{1/2} \left(\int_{0}^{1}  \exp\left(\lambda\,\sum_{m=1}^{p/2} U_{2m-1}(t+\xi)\right) \, \textrm{d}t \right)^{1/2}.
\end{equation*}
In order to estimate these two integrals, we properly adapt the ideas and results of \cite{Berkes}, together with an observation 
due to Hartman \cite{Hartman} and a lemma of an article by Takahashi \cite{Tak}. One first key step is to use the inequality
$e^x \leqslant (1+x+x^2)e^{|x|^3}$, $x\in \mathbb{R}$, to get
\begin{align*}
\int_{0}^{1}  \exp &\left(\lambda\,\sum_{m=0}^{p/2} U_{2m}(t+\xi)\right) \, \textrm{d}t
 \ll \int_{-\infty}^{\infty} \left( \frac{\sin t}{t} \right)^2 \exp\left(\lambda\,\sum_{m=0}^{p/2} U_{2m}(t+\xi)\right) \, \textrm{d}t \\
& \ll \exp\left(C_1 \lambda^3 p H^3 \right) 
\int_{-\infty}^{\infty} \left( \frac{\sin t}{t} \right)^2 
\prod_{m=0}^{p/2} \left( 1+ \frac{ \lambda U_{2m}(t+\xi)}2 +  \frac{ \lambda^2 U_{2m}(t+\xi)^2}4 \right) \, \textrm{d}t.
\end{align*}
Then by using the Fourier expansion of $U_{2m}(t)$ and the fact that
\[
\int_{-\infty}^{\infty} \left( \frac{\sin t}{t} \right)^2  \cos(u(t + \alpha))  \, \textrm{d}t = 0
\]
for all real $\alpha$ and real $u\geqslant 2$ if follows -- we skip the technical details -- that for some $C_2> 0$
\[
\int_{-\infty}^{\infty} \left( \frac{\sin t}{t} \right)^2 
\prod_{m=0}^{p/2} \left( 1+ \frac{ \lambda U_{2m}(t+\xi)}2 +  \frac{ \lambda^2 U_{2m}(t+\xi)^2}4 \right) \, \textrm{d}t
\ll \left( 1 + C_2 \lambda^2 H \right)^{p/2}
\]
and consequently 
\[
\int_{0}^{1}  \exp \left(\lambda\,\sum_{m=0}^{p/2} U_{2m}(t+\xi)\right) \, \textrm{d}t
\ll \exp\left({C_1}\,\lambda^3\,J\,H^2 + {C_2}\,\lambda^2\,J \right).
\]
This leads then to 
\[
\chi_g(\lambda,J) \ll \exp\left({C_1}\,\lambda^3\,J\,H^2 + {C_2}\,\lambda^2\,J \right)
\]
and by (\ref{inter2}) and (\ref{inter1}) to
\[
\int_{1}^{B}  \exp\left(- S_{J}(\tau) \right) \, {\textrm{d}}\tau \ll \exp\left(\frac{-J}{24}+\frac{\kappa}{2}\right) + \exp\left(\frac{\lambda J}{\pi^2 H^\gamma} + {C_2\,\lambda^2\,J} + {C_1\,\lambda^3\,JH^2}- \lambda\kappa \right).
\] 
By choosing $\kappa = J/25$ and $H$ and $\gamma$ sufficiently large as well as  $\lambda>0$ sufficiently small such that 
\[
\frac 1{\pi^2 H^\gamma} + C_2 \lambda + C_1 \lambda^2 H^2 \leqslant \frac 1{50}
\]
we obtain the proposed result with $\eta = \min\{ 1/600, \lambda/50 \}$.
\end{proof}

A direct application of Lemma~\ref{Leimportant} gives (where we assume without loss of generality that $B > e^\eta$)
\begin{align*}
Q_F(1/T) &\ll \frac 1T \int_{-T}^T |\varphi(t)|\, \d t \\
 &\ll \frac 1T \( 1 + \sum_{J=0}^{J_0} \( B\,\exp( - \eta) \)^J  \) \\
&\ll \frac 1T \( B\,\exp( - \eta) \)^{J_0} \\
&\ll T^{-\eta/\log B}. 
\end{align*}
Furthermore we have
\[
\sum_{j= L-h+1}^L |f(2^j)| \ll \beta^{L-h}
\]
and
\[
\sum_{j >  L} |f(2^j)| \ll \beta^{L}
\]
so that Theorem~\ref{Th2B} gives
\[
\|F - F_N\| \ll T^{-\eta/\log (1/\beta)} + T \beta^{L-h} \leqslant T^{-\eta/\log (1/\beta)} + T\, N^{-\frac{\log(1/\beta)}{\log 2}} 
(T\log T)^{\frac{\log(1/\beta)}{\log 2}}.
\]
Hence by choosing $T = N^{c_0(\beta)}$ where $c_0=c_0(\beta)=\log(1/\beta)^2 /\left(\eta\log 2 + \log(1/\beta)\log(2/\beta)\right)$ we obtain
$\|F - F_N\| \ll N^{- c(\beta) }
(\log N)^{  \log(1/\beta)/\log 2}$ with $c(\beta)=\eta\,c_0 /\log(1/\beta)$ as proposed.\\
Moreover, if $F$ is absolutely continuous, then $Q_F(1/T) \ll 1/T$ and we obtain the second upper bound.

\section{Cantor Digital Expansions}\label{sec:Cantor}


The purpose of this part is to state an effective Erd\H{o}s-Wintner theorem for Cantor numeration system. 
In fact, the results and the proofs are very similar to the previous case.

We start by recalling the construction of a Cantor numeration system: we choose a sequence $(a_n)_{n\geqslant  0}$ of integers such that $a_n \geqslant  2$ for all $n$. The so-called Cantor numeration system $Q$ is thus the sequence $(q_n)_{n\geqslant  0}$ defined by $q_0=1$ and $q_{n+1}=a_n\,q_n$ for all $n\geqslant  0$, hence $q_n = a_{n-1} \cdots a_1a_0$. Then every nonnegative integer $n$ has a unique expansion
$$ 
n=\sum_{j\geqslant  0} \delta_j(n)q_j, \qquad 0 \leqslant  \delta_j(n) \leqslant  a_j-1. 
$$
The length of $N$ is thus defined by $L=L(N):=\max\left\{j\geqslant 0 : q_j \leqslant N < q_{j+1}\right\}$.

Clearly this type of system is a generalization of the $q$-adic case by taking $a_n = q \geqslant  2$ for all $n$.
Furthermore we define $Q$-additive function by
\[
f( d_1 q_{e_1} + \cdots + d_r q_{e_r}) = f(d_1 q_{e_1}) + \cdots + f(d_r q_{e_r}),
\]
where $e_1< e_2 < \cdots < e_r$ are different positive integers and
$d_1,\ldots, d_r$ are integers digits satisfying $1\leqslant d_j \leqslant q_j-1$ ($1\leqslant j\leqslant r$). 

For example, the Van-der-Corput sequence 
$v_Q(n)$  related to the Cantor numberation system $Q$ is defined by $v_Q(d q_a) = dq_{a+1}^{-1}$. 

A Cantor numeration system is said to be \textit{constant-like} if the sequence $(a_n)$ is bounded.

\subsection{Erd\H{o}s-Wintner theorem for constant-like Cantor numeration system}

In \cite{Co} Coquet proved an Erd\H{o}s-Wintner theorem for $Q$-additive functions when $(a_n)$ is bounded.

\begin{theo}\label{ThCantor}
	Let $f(n)$ be a real valued $Q$-additive function with respect to a constant-like Cantor numeration system. Then 
	$f(n)$ has a distribution function $F(y)$, that is 
	\begin{equation}\label{eqC}
	\lim_{N\to\infty}\frac 1N \#\left\{n<N\mid f(n) \leqslant y\right\}=F(y),
	\end{equation}
	if and only if the two series 
	\begin{equation}\label{eqC2}
	\sum_{j\geqslant  0} \frac{1}{a_j}\sum_{d=1}^{a_j-1}f(dq_j)\quad\text{and}\quad
	\sum_{j\geqslant  0}\frac{1}{a_j}\sum_{d=1}^{a_j-1}f(dq_j)^2
	\end{equation}
	converge. In this case the characteristic function $\varphi(t)$ of the limiting distribution is given by
	\begin{equation}\label{eqvarphiC}
	\varphi(t) = \int_{-\infty}^\infty e^{ity} \; \d F(y) =
	\prod_{j\geqslant  0} \left( \frac{1}{a_j}  \sum_{d=0}^{a_j-1} e^{it f(d q_j)}  \right).
	\end{equation}
\end{theo}

It can be proved in the same way as in the  $q$-adic (by using L\'evy's theorem). We note that
Barat and Grabner \cite{BaratGrabner} wrote an alternative proof from ergodic point of view 
where they could avoid ``the original Fourier analysis techniques''.

We also note that distribution measure given by $F(y)$ is pure and that it consists only of finitely many 
point masses if and only if there exists a $J$ such that $f(dq_j) = 0$ for $j > J$ and all $d\in \{1,\ldots, a_j-1\}$
(see \cite[Proposition 3]{BaratGrabner}).

\subsection{An Effective Erd\H{o}s-Wintner theorem for constant-like Cantor numeration system}

As in the $q$-adic case, if only finitely many values $f(d q_j)$ are non-zero, then $f(n)$ is 
periodic with period $q_J$ and by (\ref{eqvarphiC}) the limiting distribution function $F$ equals $F_{q_J}$ and we obtain
\[
\|F_N- F\|_\infty  \ll \frac 1N
\]
which is (again) the optimal convergence rate.

From now on we will assume that $f(dq_j) \ne 0$ for infinitely many instances. 
By doing the same reasoning as for the proof we obtain the following result.

\begin{theo}\label{Th2C}
	Let $f(n)$ be a real-valued $Q$-additive function {\normalfont (}with respect to a constant-like Cantor numeration system{\normalfont )} such that the two series {\normalfont (\ref{eqC2})} converge and  $f(dq_j) \ne 0$ for infinitely many instances. 
	Let $L=L(N)$ denote the length of $N$ and $h = \lceil \log_a( T \log T ) \rceil$
 {\normalfont (}where $a$ is the minimum of $(a_n)${\normalfont )}. 
Then we have for all real numbers $T\geqslant 1$ such that $h \leqslant L(N)$
\begin{equation} \label{eqTh2C-ext}
\|F_N - F\|_{\infty} \ll Q_F\left(\frac{1}{T}\right)
+ T \left| \sum_{j > L} \frac 1{a_j}  \sum_{d=0}^{a_j-1} f(dq_j) \right| +
T \sqrt h \sqrt{  \sum_{j\geqslant L-h} \frac 1{a_j} \sum_{d=0}^{a_j-1} f(dq_j)^2 },
\end{equation}
where $F_N$ denotes the distribution function of $\(f(n): n<N\)$ and $F$ the limiting distribution function.
All implicit constants depend on the minimum and maximum of the sequence $(a_n)_n$.
\end{theo}

We note that the same remarks as given for Theorems~\ref{Th2A} and \ref{Th2B} apply here, too. 
In particular we can state an analogue to Theorem~\ref{Th2B}: 
	\begin{align*}
	\|F_N - F\|_{\infty} &\ll \frac 1T \int_{0}^T  \exp\left( - c_2\,t^2 \sum_{(d,j)\in S_C(t)} \frac{f(d q_j)^2}{a_j^2} \right)\, \d t 
	+ T  \sum_{j=L-h+1}^L \frac{1}{a_j} \sum_{d=0}^{a_j-1} |f(dq_j)|\\
	&+ \int_{1/T}^T \min\left\{ \frac 1{1+t}, \sum_{j >  L} \frac{1}{a_j} \sum_{d=0}^{a_j-1} |f(dq_j)| \right\} 
	\exp\left( - c_2 \, t^2 \sum_{(d,j)\in S_C(t),\, j \leqslant  L} \frac{f(d q_j)^2}{a_j^2} \right) \, \d t,
	\end{align*}
where
	\[
	S_C(t) = \{(d,j) \in \{1,\ldots,a_j-1\}\times \mathbb{N} : |f(dq_j)| \leqslant \pi/ |t| \}
	\]
and $c_2= 2/\pi^2$.
We do not give the details, we only note that we used the fact that $(a_n)$ is bounded 
to obtain a result similar to Lemma \ref{LeO1} (which explains the dependence on the sequence $(a_n)_n$ in Theorem~\ref{Th2C}). 
And we also state the following result that is analogous to Lemma \ref{Ledifference}.

\begin{lemma}\label{LedifferenceC}
	Suppose that $f(n)$ is real-valued and $Q$-additive. Then we have for all integers $h\geqslant  1$ and $N \geqslant  q_{\lceil h \rceil}$
	\begin{equation}\label{eqLedifferenceC}
	|\varphi_N(t) - \varphi_{q_{L+1}}(t)| \leqslant \frac 2{a_{L-h+1} \cdots a_{L-1}} + 2\sqrt 2  \sum_{j=L-h+1}^L  
\max_{1\le d  \le a_j-1} (1-\cos(t f(dq_j))^{1/2}
	\end{equation}
	where $L=L(N)$ is the length of $N$.
\end{lemma}

\begin{ex}
The $Q$-additive function
\begin{eqnarray*}
	v_Q(n) = \sum_{k=0}^{\infty} \frac{\delta_k(n)}{q_{k+1}}
\end{eqnarray*}
defines the van der Corput sequence for a Cantor numeration system $(a_n)_n$. Let assume that the sequence $(a_n)_n$ is bounded.
Then the two series (\ref{eqC2}) certainly converge. It is also possible to compute the characteritic function of
the limiting distribution
\begin{align*}
	\Phi(t) &= \prod_{j=0}^{\infty} \frac{1}{a_j}\left(1+\sum_{d=1}^{a_j-1} \exp\left(\frac{it d}{q_{j+1}}\right)\right) \\
		&= \prod_{j=0}^{\infty} \left( \frac{\exp\left(\frac{it }{2q_{j}} -\frac{it }{2q_{j+1}}\right)}{a_j} \, \frac{\sin\left(t / 2q_j\right)}{\sin\left(t / 2q_{j+1}\right)}\right)\\
	&= e^{\frac{it}{2}}\,\prod_{j=0}^{\infty} \frac{2q_j\,\sin\left(t / 2q_j\right)}{2q_{j+1}\,\sin\left(t / 2q_{j+1}\right)} \\
	&= e^{\frac{it}{2}}\;\frac{\sin\left(t/2\right)}{t/2},
\end{align*}
that is precisely the characteristic function of the uniform distribution on $[0,1]$. In particular, $F$ is absolutely continuous
with density $F'(z) = 1$ (for $0< z < 1$). In particular $Q_F(h) = \min(h,1)$.

This corresponds to the well-known fact that the Van-der-Corput sequence is uniformly distribution modulo 1.
The distance $\left\|F-F_N\right\|_{\infty}$ is then precisely the discrepancy, for which it is known that 
\[
\left\|F-F_N\right\|_{\infty} \ll \frac{\log\,N}{N}.
\]
We can find more general and specific results in \cite{HLN}, for example. By applying the methods used in 
the proof of  Theorem \ref{Th2C}, we 
only obtain
\[
 \left\|F-F_N\right\|_{\infty} \ll \frac{\log\,N}{N^{1/3}}
\]
where we have used that fact that $Q_F(h) = \min(h,1)$ and where we choose $T = N^{1/3}$.
\end{ex}

\subsection{A Partial Erd\H{o}s-Wintner Theorem for General Cantor Numeration Systems}

It is a natural question what can be said for unbounded sequences $(a_n)$. In \cite{Co}, Conquet stated a sufficient 
condition for the existence of a limit distribution.

\begin{prop}[Coquet \cite{Co}]
	Let $f$ be a real-valued $Q$-additive function. We set, for all $j\geqslant 0$ and $d\in\{1,...,a_j-1\}$
	\begin{eqnarray*}
		f^{\ast}(dq_{j}) = \left\{
		\begin{array}{ll}
			f(dq_j), & \textrm{ if } \, |f(dq_j)| \leqslant 1,\\
			1, & \textrm{ if } \, |f(dq_j)| > 1,
		\end{array}
		\right.
	\end{eqnarray*}
	and
	$$ \beta_j := \max_{1 \leqslant k \leqslant a_j-1} \left(\frac{1}{k+1}\,\sum_{d=0}^{k} f^{\ast}(dq_{j})\right)^2.$$
	\\
	If $\beta_j \rightarrow 0$, and the series
	$$ \sum_{j\geqslant 0} \frac{1}{a_j} \sum_{d=1}^{a_j-1} f^{\ast}(dq_j) \;\;\;\; \text{and} \;\;\;\; \sum_{j\geqslant 0} \frac{1}{a_j} \sum_{d=1}^{a_j-1} f^{\ast}(dq_j)^2$$
	converge, then $f$ has a limit distribution and its characteristic function is
	$$ \varphi(t) = \prod_{j=0}^{\infty} \left( \frac{1}{a_j}  \sum_{d=0}^{a_j-1} \exp\left(it f(dq_j)\right)\right). $$
\end{prop}

We already mentioned that  Barat and Grabner \cite{BaratGrabner} proved Theorem~\ref{ThCantor} with the help of ergodic means.
Actually they relate it to a convergence property of a series $\sum f_n(x)$ that is defined on the
$Q$-adic integers $\displaystyle x \in \mathbb{Z}_Q = \lim_{\leftarrow} \mathbb{Z}/q_n \mathbb{Z}$. 
However, in the non-constant-like case they observe (with the help of a counter-example) that this relation
is no longer an equivalence. It remains an open problem to formulate an
Erd\H os-Wintner theorem for general Cantor systems.

\section{Zeckendorf Digital Expansions}\label{sec:Zeckendorf}

For $k \geqslant  0$ let $F_k$ be the $k$-th Fibonacci number, that is, $F_0 = 0$, $F_1 = 1$ and
$F_k = F_{k-1} + F_{k-2}$ for $k \geqslant  2$. By Zeckendorf's theorem every positive integer
$n$ admits a unique representation
\[
n = \sum_{i=2}^L \delta_i(n) F_i,
\]
where $\delta_i(n) \in \{0, 1\}$ and $\delta_i(n) =1$ implies $\delta_{i+1}(n)  = 0$;
$L$ denotes the length of this expansion.
This is by the way the optimal representation of $n$ as the sum of Fibonacci numbers in the sense that
the number of Fibonacci numbers is minimal. We also recall that the Fibonacci numbers are 
explicitly given by
\[
F_k = \frac 1{\sqrt 5}\left( \gamma^k + (-1)^{k-1} \gamma^{-k} \right),
\]
where $\gamma = (1+\sqrt 5)/2$ is the golden number that satisfies the equation $\gamma^2 = 1 + \gamma$. 

A function $f$ on the non-negative integers is called Z-additive if
\[
f(n) = \sum_{i=2}^L f(\delta_i(n) F_i),
\]
that is, $f$ is uniquely determined by $f(0) = 0$ and the values $f(F_i)$, $i\geqslant  2$.

\subsection{The Zeckendorf Erd\H os-Wintner theorem}

Our first result is a proper version of the Erd\H os-Wintner theorem for 
Z-additive functions. We note that a partial result was given by \cite{BG1}.
\begin{theo}\label{ThZeckendorf}
Let $f(n)$ be a real-valued  Z-additive function. Then 
$f(n)$ has a distribution function $F(y)$
if and only if the two series 
\begin{equation}\label{eqThZeckendorf}
\sum_{j\geqslant  2} f(F_j)\quad\text{and}\quad
\sum_{j\geqslant  2} f(F_j)^2
\end{equation}
converge. In this case the characteristic function $\varphi(t)$ of the limiting distribution
is given by 
\begin{equation}\label{eqThZphirep}
\varphi(t) = \frac{\sqrt 5}{\gamma}  \prod_{j\geqslant  2} \frac{r_j(t)}\gamma,
\end{equation}
where $r_2(t) = 1$ and 
\begin{equation}\label{eqThZrjrec}
r_{j+1}(t) = 1 + \frac{ e^{it f(F_j)}}{r_j(t)} \qquad (j\geqslant  2)
\end{equation}
provided that $r_j(t)\ne 0$ for all $j\geqslant  2$.
\end{theo}

We note that the limiting distribution is purely atomic if and only if $f(F_j) = 0$ for $j\ge J$
(for some integer $J$), see \cite[Proposition 11]{BaratGrabner}.

As we will see in the discussion below it might happen that $r_j(t) = 0$ for finitely many $j$
and that we have then a similar infinite product representation for $\varphi(t)$. 
In any case the representation (\ref{eqThZphirep}) is valid for sufficiently small $t$.

Set 
\[
H_k(t) := \sum_{n < F_k} e^{it f(n)}
\]
Then by defintion we have $H_1(t) = H_2(t) = 1$ and
\begin{equation}\label{eqHkrec}
H_k(t) = H_{k-1}(t) + e^{it f(F_{k-1})} H_{k-2}(t) \qquad (k \geqslant  3).
\end{equation}
Furthermore if we set 
\begin{equation}\label{eqAkdef}
A_k(t) = \left( \begin{array}{cc} 1 & e^{itf(F_k)} \\ 1 & 0 \end{array} \right)
\end{equation}
then (\ref{eqHkrec}) rewrites to
\[
\left( \begin{array}{c}  H_k(t) \\ H_{k-1}(t)  \end{array} \right) = A_{k-1}(t) \left( \begin{array}{c}  H_{k-1}(t) \\ H_{k-2}(t)  \end{array} \right) 
\]
and consequently we have
\[
H_k(t) = \left( \begin{array}{cc}  1 & 0 \end{array} \right) A_{k-1}(t) A_{k-2}(t) \cdots A_2(t) 
\left( \begin{array}{c}  1 \\ 1  \end{array} \right). 
\]
Next we assume that $r_j(t)\ne 0$ for all $j\geqslant  2$. 
We observe that the recurrence (\ref{eqThZrjrec}) is equivalent to the relation
\[
A_k(t) \left( \begin{array}{c}  r_k(t) \\ 1  \end{array} \right) = r_k(t)
\left( \begin{array}{c}  r_{k+1}(t) \\ 1  \end{array} \right)
\]
which leads directly to
\[
H_k(t) = r_2(t)r_3(t) \cdots r_{k-1}(t) r_k(t).
\]
Thus, if we have a limiting distribution then the characteristic function of the limit is given
\begin{equation}\label{eqThZphirep2}
\varphi(t) = \lim_{k\to\infty} \frac 1{F_k} H_k(t) = \lim_{k\to\infty} \frac {\sqrt 5}{\gamma^k} H_k(t) 
= \frac{\sqrt 5}{\gamma}  \prod_{j\geqslant  2} \frac{r_j(t)}\gamma.
\end{equation}

Next let us assume that $f(F_k)\to 0$ as $k\to\infty$. (This is certainly implied by (\ref{eqThZeckendorf}).)
By (\ref{eqHkrec}) this implies that for every fixed real number $t_0>0$ we have $r_k(t) = H_k(t)/H_{k-1}(t) \to \gamma$
uniformly for $|t|\leqslant t_0$. The converse is also true (here we have use again Lemma~\ref{LeAppendix} of the Appendix). 
In particular this shows that for every fixed $t_0>0$ there exists $j_0$ such that 
$r_j(t)\ne 0$ for $|t|\leqslant t_0$ and  $j\geqslant  j_0$. In the same way as above we obtain
\[
\varphi(t) = \frac{\sqrt 5}{\gamma^{j_0}} H_{j_0}(t)\prod_{j\geqslant  j_0+1}  \frac{r_j(t)}\gamma \qquad (|t|\leqslant t_0),
\]
where $r_j(t)$ satisfies the same recurrence (\ref{eqThZrjrec}) as above for $j \geqslant j_0$, and
where we can compute $r_{j_0+1}(t)$ by
\[
r_{j_0+1}(t) = \frac{H_{j_0+1}(t)}{H_{j_0}(t)}.
\]

Next we will study the recurrence (\ref{eqThZrjrec}) in more detail. For this purpose we use
the following notation:
\[
\eta_k = \eta_k(t) := e^{itf(F_k)}-1, \qquad \varepsilon_k = \varepsilon_k(t) := r_k(t)- \gamma.
\]
Then (\ref{eqThZrjrec}) rewrites to 
\begin{equation}\label{eqepsilonrec}
\varepsilon_{k+1} = \frac{\eta_k - (\gamma-1) \varepsilon_k}{\gamma + \varepsilon_k}.
\end{equation}
Note that -- for notational simplicity -- we skip the dependence on $t$.
Further note that $r_k(t) \to \gamma$ is equivalent to $\varepsilon_k \to 0$.

In order to quantify the above considerations on the non-zeroness of $r_k(t)$ we note
that $|\eta_k| = |e^{itf(F_k)}-1| \leqslant 2\gamma - 3$ and $|\varepsilon_k| \leqslant \gamma -1$ 
implies $|\gamma + \varepsilon_k|\geqslant  1$ and consequently
\[
|\varepsilon_{k+1}| \leqslant  |\eta_k| + (\gamma-1) |\varepsilon_k| \leqslant 2\gamma - 3 + (\gamma-1)^2 = \gamma -1.
\]

\vspace*{0.5cm}

\begin{lemma}\label{Leequiv}
Suppose that $f(F_k)\to 0$ as $k\to\infty$. Then the condition that the two series { \normalfont (\ref{eqThZeckendorf})} converge 
is equivalent to the statement that the two series
\begin{equation}\label{eqLeequiv}
\sum_{j\geqslant  j_0} \varepsilon_j \quad \mbox{and}\quad \sum_{j\geqslant  j_0} |\varepsilon_j|^2
\end{equation}
converge, where $j_0$ is chosen in a way that $\varepsilon_j \ne - \gamma$ for $j\geqslant  j_0$.
\end{lemma}

\begin{proof}
Since $f(F_j)\to 0$ we also have that $\varepsilon_j \to 0$, and we can assume that
$|tf(F_j)|\leqslant \pi$ and $|\varepsilon_j|\leqslant \frac 12$ for $j\geqslant  j_0$.
By using the inequalities $\frac 4{\pi^2} |x| \leqslant |e^{ix} - 1| \leqslant |x|$ (for $|x|\leqslant \pi$) it directly
follows that 
\[
\sum_{j\geqslant  0} |\eta_j|^2 < \infty \quad\mbox{if and only if} \quad  \sum_{j\geqslant  j_0} f(F_j)^2 < \infty.
\]
and by applying the expansion $e^{ix}-1 = ix + O(x^2)$  (for $x\to 0$)
it also follows that the series $\sum_j \eta_j$ is convergent if and only if
the series $\sum_j f(F_j)$ is convergent. 

Since $|\varepsilon_j|\leqslant \frac 12$ it follows that $|\gamma + \varepsilon_j| \geqslant  \gamma -\frac 12$,
and with $L = (\gamma - 1)/(\gamma - \frac 12) < 1$ we have
\[
|\varepsilon_j| \leqslant |\eta_{j-1}| + L\, |\varepsilon_{j-1}|
\]
and by iteration
\[
|\varepsilon_j| \leqslant \sum_{\ell=1}^{j-j_0} |\eta_{j-\ell}| L^{\ell-1} + L^{j-j_0}\, |\varepsilon_{j_0}|.
\]
Consequently
\begin{align*}
\sum_{j\geqslant  j_0} |\varepsilon_j|^2 &\leqslant 2 \sum_{j\geqslant  j_0}  \sum_{k,\ell=1}^{j-j_0} |\eta_{j-\ell}||\eta_{j-k}| L^{k+\ell-2} 
+ 2 \sum_{j\geqslant  j_0} L^{2(j-j_0)}\, |\varepsilon_{j_0}|^2 \\
&= 2\sum_{k,\ell \, \geqslant  1} L^{k+\ell-2}  \sum_{j\geqslant  j_0 + \max\{k,\ell\}}   |\eta_{j-\ell}||\eta_{j-k}| +
2\,\frac{|\varepsilon_{j_0}|^2}{1-L^2} \\
&\leqslant \frac 1{(1-L)^2} \sum_{j\geqslant  j_0} |\eta_j|^2 + 2\, \frac{|\varepsilon_{j_0}|^2}{1-L^2},
\end{align*}
where we have used the Cauchy-Schwarz inequality to derive
\[
\sum_{j\geqslant  j_0 + \max\{k,\ell\}}   |\eta_{j-\ell}||\eta_{j-k}|  \leqslant \sum_{j\geqslant  j_0} |\eta_j|^2.
\]
Thus, if the series $\sum_j  |\eta_j|^2$ is convergent, then the series $\sum_j  |\varepsilon_j|^2$ converges, too.

The converse statement is much easier to prove. From (\ref{eqepsilonrec}) and the assumption $|\varepsilon_j|\leqslant \frac 12$
we obtain 
\begin{equation}\label{eqetajrep}
\eta_j = \frac{\varepsilon_j}\gamma + \gamma \varepsilon_{j+1} + \varepsilon_{j}\varepsilon_{j+1}
\end{equation}
and consequently
\[
|\eta_j| \leqslant c_1 |\varepsilon_{j}| + c_2 |\varepsilon_{j+1}| 
\]
for some positive constants $c_1,c_2$. Hence, if the series $\sum_j  |\varepsilon_j|^2$ is convergent, the same
property holds for the series $\sum_j  |\eta_j|^2$.

Next assume that the two series $\sum_j \eta_j$ and $\sum_j |\eta_j|^2$ converge. As argued above this implies
that the series $\sum_j  |\varepsilon_j|^2$ converges, too. By the Cauchy-Schwarz inequality this also implies
that the series $\sum_j \varepsilon_{j}\varepsilon_{j+1}$ converges.
Hence, by using the fact that $\varepsilon_{j} \to 0$ (as $j\to\infty$) and the relation { \normalfont (\ref{eqetajrep})} it follows that the
sum 
\[
\sum_{j=L}^M \varepsilon_{j} = \left( \frac 1\gamma + \gamma\right)^{-1} 
\left( \sum_{j=L}^M \eta_{j}  - \sum_{j=L}^M \varepsilon_{j}\varepsilon_{j+1}  + \varepsilon_{L} - \varepsilon_{M+1} \right)
\]
can be made arbitrarily small and, thus, Cauchy's criterion implies the convergence of the
series $\sum_{j} \varepsilon_{j}$.

Again the converse implication is easier to obtain. If we assume that the two series $\sum_j \varepsilon_j$
and $\sum_j  |\varepsilon_j|^2$ converge then by using (\ref{eqetajrep}) it directly follows that 
the series  $\sum_j \eta_j$ converges. We just recall that the series  $\sum_j |\eta_j|^2$ converges, too.
Summing up this completes the proof of the lemma.
\end{proof}

We are now in the position to prove Theorem~\ref{ThZeckendorf}.
Suppose first that the two series (\ref{eqThZeckendorf}) converge. Then by Lemma~\ref{Leequiv} the two series (\ref{eqLeequiv}) converge, too, which implies that the infinite product
\[
\prod_{j\geqslant  j_0+1}  \frac{r_j(t)}\gamma = \prod_{j\geqslant  j_0+1} \left( 1 + \frac{\varepsilon_k(t)}{\gamma} \right)
\]
converges.  Thus, by (\ref{eqThZphirep2}) the limit
\[
\varphi(t) = \lim_{k\to \infty} \varphi_{F_k}(t)
\]
exists (where $\varphi_N(t) := (1/N)\,\sum_{n<N} \exp(i\,t\,f(n))$). As in the $q$-adic case, this also implies that $\varphi_N(t) \to \varphi(t)$ (see Lemma~\ref{Ledifference2}).
Finally, the limiting function $\varphi(t)$ is continuous at $t=0$ (for example, Lemma~\ref{LeZZ} implies
that $\varphi(t) = 1 + O(t)$ as $t\to 0$) and the conclusion follows thanks to L\'evy's theorem.\\

Now we assume that $f(n)$ has a limiting distribution. 
This implies that $H_k(t)/F_k \to \varphi(t)$, where
$\varphi(t)$ is the characteristic function of the limiting distribution.
By (\ref{eqHkrec}) this implies
that $e^{itf(F_k)} \to 1$ as $k\to\infty$ and consequently (also by using Lemma~\ref{LeAppendix} in Appendix)
$f(F_k)\to 0$ as $k\to\infty$. 

The convergence to $\varphi(t)$ can be also rewritten to 
\begin{equation}\label{eqvarphiproduct}
\varphi(t) = \frac{\sqrt 5}\gamma  \lim_{k\to\infty} \frac 1{\gamma^{k-1}}  
\left( \begin{array}{cc}  1 & 0 \end{array} \right) A_{k}(t) A_{k-1}(t) \cdots A_2(t) 
\left( \begin{array}{c}  1 \\ 1  \end{array} \right).
\end{equation}
Let $\|A\|_2 := \sqrt{\rho(AA^*)}$ denote the spectral norm of a matrix $A$ (here $\rho$ denotes the
spectral radius and $A^*$ the Hermitian conjugate). It is easy to see that $\| A_j(t)\|_2 = \gamma$. 
However we have $\|A_{j+1}(t) A_{j}(t)\|_2 \leqslant \gamma^2 \exp(- c |\eta_{j+1}|^2)$ for some constant $c> 0$ since
\begin{align*}
A_{j+1}(t) A_{j}(t)(A_{j+1}(t) A_{j}(t))^* &= \left( \begin{array}{cc}  5 + \eta_{j+1} + \overline \eta_{j+1} & 3 + \eta_{j+1} \\ 
3 + \overline \eta_{j+1} & 2 \end{array} \right) \\ &= 
\left( \begin{array}{cc}  5 - 2(1-\cos(t f(F_{j+1}))) & 3 -( 1 - e^{itf(F_{j+1})}) \\ 
3 -( 1 - e^{-itf(F_{j+1})}) & 2 \end{array} \right)
\end{align*}
and the spectral radius satisfies
\[
\rho\left(\left( \begin{array}{cc}  5 - 2(1-\cos(t f(F_{j+1}))) & 3 -( 1 - e^{itf(F_{j+1})}) \\ 
3 -( 1 - e^{-itf(F_{j+1})}) & 2 \end{array} \right) \right) \leqslant \gamma^4 - c' (1-\cos(t f(F_{j+1}))) \leqslant \gamma^4 - c'' |\eta_{j+1}|^2
\]
for proper constants $c'>0$, $c''> 0$.
Hence by taking norm at (\ref{eqvarphiproduct}) and by grouping consecutive matrices together we obtain 
(similarly to the proof in the $q$-adic case) that 
\[
\frac 12 \leqslant |\varphi(t)| \leqslant \min\left\{ \exp\left( - c \sum_{j\geqslant  1} \left\| \frac{t f(F_{2j})}{2\pi} \right\|^2  \right),
\exp\left( - c \sum_{j\geqslant  1} \left\| \frac{t f(F_{2j+1})}{2\pi} \right\|^2  \right) \right\}
\]
which implies (as in the $q$-adic case) that the series $\sum_j f(F_j)^2 $ converges.
At this point we can now argue as in the proof of Lemma~\ref{Leequiv} and observe that
the series $\sum_j |\varepsilon_j(t)|^2$ converges, too.
We choose $t>0$ sufficiently small such that the representation (\ref{eqThZphirep}) holds.
Thus, the series $\sum_j \varepsilon_j(t)$ converges and consequently by Lemma~\ref{Leequiv}
the two series (\ref{eqThZeckendorf}) converge which completes the proof of Theorem~\ref{ThZeckendorf}.

\subsection{An Effective Version of the Zeckendorf Erd\H os-Wintner theorem}

Similarly to Theorem~\ref{Th2A} we can formulate a quantitative version for 
Zeckendorf additive functions that are, however, slightly weaker. 
\begin{theo}\label{ThZ2}
Suppose that $f(n)$ is a real-valued Z-additive function such that the series
\begin{equation}\label{eqserabsvalFj}
\sum_{j\geqslant 2} |f(F_j)|
\end{equation}
converges. Set $L = \lceil \log_\gamma (\sqrt 5 \, N )\rceil$ and $h = \lceil \log_\gamma( T \log T) \rceil$.
Then we have for all real numbers $T\geqslant 1$ such that $h \leqslant L/2$
\begin{equation}
\|F_N - F\|_{\infty} \ll Q_F(1/T)  + \frac{\log\,N}{T} +  T  \sum_{j > L-2h+1}  |f(F_j)|   \label{eqTh2Z-ext}
\end{equation}
where $F_N$ denotes the distribution function of $(f(n): n<N)$ and $F$ the limiting 
distribution function.
\end{theo}

We note that there is also an analogue to Theorem~\ref{Th2B} that is of the form
\begin{align}
\|F_N - F\|_{\infty} &\ll \frac 1T \int_{0}^T  \exp\left( - c_2 \, t^2 \sum_{j \in \tilde S(t)} f(F_j)^2 \right)\, \d t + \frac{\log N}T
+ T \,  \sum_{j=L-2h+1}^L |f(F_j)| \label{eqTh2Z-2} \\
&+ \int_{1/T}^T \min\left\{ \frac 1{1+t}, \sum_{j\geqslant  L-h}  |f(F_j)| \right\} 
\exp\left( - c_2 \, t^2 \sum_{ j \in \tilde S(t),\, j <  L-h} f(F_j)^2 \right) \, \d t,  \nonumber
\end{align}
where
\[
\tilde S(t) = \{ j\geqslant  2 : |f(F_j)| \leqslant \pi/ |t| \}
\]
and $c_2>0$ is some constant.

\begin{nrem}\label{uprem}
	By choosing $h = \lceil \log_\gamma( T \log\,N \log T) \rceil$ in the two previous upper bounds, the term $\log(N)/T$ disappears.
\end{nrem}

The main differences between Theorems~\ref{Th2A} and \ref{Th2B} and Theorem~\ref{ThZ2} are
the additional term $(\log N)/T$ and the need of the assumption (\ref{eqserabsvalFj}), at least for
the upper bound (\ref{eqTh2Z-ext}) (the upper bound (\ref{eqTh2Z-2}) works in principle in all cases).
Both seem to be artefacts of the proof, however, it seems that
the current proof methods are not strong enough to overcome these artefacts. 
Nevertheless, the proof uses quite the same ideas as that of the previous theorem. 
The main technical difficulty is to handle non-commutative matrix products.

We start with two lemmas
that are analogues to  Lemma~\ref{LeO1} and to Lemma~\ref{Ledifference}.
\begin{lemma}\label{LeZZ}
Let $f(n)$ be a real-valued Z-additive function such that the two series {\normalfont(\ref{eqThZeckendorf})} converge.
 Then we have
\begin{equation}\label{eqLeZ}
\frac 1N \sum_{n< N} |f(n)| = O(1)
\end{equation}
as $N\to \infty$.
\end{lemma}

\begin{proof}
We choose $F_L$ such that $F_{L-1} <N\leqslant F_L$ and will prove that 
\[
\sum_{n< F_L} |f(n)| = O(F_L).
\]
Clearly this implies (\ref{eqLeZ}). 

For this purpose we consider first the sums 
\[
S_k := \sum_{n< F_k} f(n).
\]
By definition of $f$ we obtain the recurrence
\begin{align*}
S_k &= \sum_{n< F_{k-1}} f(n) + \sum_{n'< F_{k-2}} f(F_{k-1} + n') \\
&= S_{k-1} + S_{k-2} + F_{k-2} f(F_{k-1})
\end{align*}
which leads (by induction) to 
\begin{equation}\label{eqSkrep}
S_k = \sum_{\ell = 2}^{k-1} F_{k-\ell} F_{\ell-1} f(F_\ell).
\end{equation}
The representation can be used to prove $S_k = O(F_k)$ but we will (\ref{eqSkrep}) directly.

Similarly we can treat the sum of squares. From the recurrence
\begin{align*}
T_k &:= \sum_{n< F_k} f(n)^2 = \sum_{n< F_{k-1}} f(n)^2 + \sum_{n'< F_{k-2}} f(F_{k-1} + n')^2 \\
&= S_{k-1} + S_{k-2} + F_{k-2} f(F_{k-1})^2 + 2 S_{k-2} f(F_{k-1}) 
\end{align*}
we obtain (again) by induction
\[
T_k = \sum_{\ell = 2}^{k-1} F_{k-\ell} F_{\ell-1} f(F_\ell)^2 + 2 \sum_{\ell = 2}^{k-1} F_{k-\ell} S_{\ell-1} f(F_\ell).
\]
The first sum can be easily handled:
\[
\sum_{\ell = 2}^{k-1} F_{k-\ell} F_{\ell-1} f(F_\ell)^2  = 
O\left( \gamma^{k-1} \sum_{\ell = 2}^{k-1} f(F_\ell)^2 \right) = O(F_k).
\]
For the second sum we use (\ref{eqSkrep}) and obtain
\begin{align*}
2 \sum_{\ell = 2}^{k-1} F_{k-\ell} & S_{\ell-1} f(F_\ell) = 
2 \sum_{\ell = 2}^{k-1} \sum_{j=2}^{\ell-2} F_{k-\ell} F_{\ell-j-1} F_{j-1}  f(F_\ell) f(F_j) \\
&= \frac 2{5\sqrt 5} \sum_{\ell = 2}^{k-1} \sum_{j=2}^{\ell-2} 
\left( \gamma^{k-\ell} + (-1)^{k-\ell+1} \gamma^{\ell-k} \right) \\
& \qquad \qquad \qquad \qquad \times
\left( \gamma^{\ell-j-1} + (-1)^{\ell-j} \gamma^{1+j-\ell} \right)
\left( \gamma^{j-1} + (-1)^{j} \gamma^{1-j} \right) f(F_\ell) f(F_j) \\
&=2 \frac{\gamma^{k-2}}{5\sqrt 5}  \sum_{\ell = 2}^{k-1} \sum_{j=2}^{\ell-2} f(F_\ell) f(F_j) 
+ \frac 2{5\sqrt 5} \sum_{\ell = 2}^{k-1} \sum_{j=2}^{\ell-2} \gamma^{k-2j} (-1)^{j} f(F_\ell) f(F_j) \\
&+  \frac 2{5\sqrt 5}\sum_{\ell = 2}^{k-1} \sum_{j=2}^{\ell-2} \gamma^{k-2\ell+2j} (-1)^{\ell-j} f(F_\ell) f(F_j) 
+ \frac 2{5\sqrt 5} \sum_{\ell = 2}^{k-1} \sum_{j=2}^{\ell-2}  \gamma^{2\ell-k-2} (-1)^{k-\ell+1} f(F_\ell) f(F_j) \\
&+ \frac 2{5\sqrt 5} \sum_{\ell = 2}^{k-1} \sum_{j=2}^{\ell-2}  \gamma^{k-2\ell+2} (-1)^{\ell}  f(F_\ell) f(F_j) 
+ \frac 2{5\sqrt 5} \sum_{\ell = 2}^{k-1} \sum_{j=2}^{\ell-2}  \gamma^{2\ell-2j-k}  (-1)^{k-\ell+j+1} f(F_\ell) f(F_j) \\
&+ \frac 2{5\sqrt 5} \sum_{\ell = 2}^{k-1} \sum_{j=2}^{\ell-2}  \gamma^{2j-k} (-1)^{k-j+1} f(F_\ell) f(F_j) 
+ \frac 2{5\sqrt 5} \sum_{\ell = 2}^{k-1} \sum_{j=2}^{\ell-2}  \gamma^{2-k} (-1)^{k+1}   f(F_\ell) f(F_j).
\end{align*}
Since
\begin{align*}
R := \sum_{\ell=2}^{k-1} \sum_{j=2}^{\ell-2} f(F_\ell) f(F_j) =
\frac 12 \left( \sum_{j=2}^{k-1} f(F_j) \right)^2 
- \frac 12 \sum_{j=2}^{k-1} f(F_j)^2 -  \sum_{j=3}^{k-1} f(F_j) f(F_{j-1})
\end{align*}
and 
\[
\left|\sum_{j=3}^{k-1} f(F_j) f(F_{j-1})\right| \leqslant \left( \sum_{j=3}^{k-1} f(F_j)^2 \cdot \sum_{j=3}^{k-1} f(F_{j-1})^2 \right)^{1/2} = O(1)
\]
it follows that $R = O(1)$ and, thus, the first part of the sum is of order $O(F_k)$.
The other parts can be handled even more directly. For the sake of shortness we only discuss the
next three terms (that are also asymptotically most significant). The remaining four terms are 
really easy to bound.
First we have 
\begin{align*}
\sum_{\ell = 2}^{k-1} \sum_{j=2}^{\ell-2} \gamma^{k-2j} (-1)^{j} f(F_\ell) f(F_j) &=
\sum_{j=2}^{k-3} \gamma^{k-2j} (-1)^{j} f(F_j) \sum_{\ell = j+2}^{k-1}  f(F_\ell) \\
&= \sum_{j=2}^{k-3} \gamma^{k-2j} |f(F_j)|\, O(1) \\
&= O\left(  \left( \sum_{j=2}^{k-3} \gamma^{2k-4j}  \sum_{j=2}^{k-3} f(F_j)^2 \right)^{1/2}\right) \\
&= O(F_k).
\end{align*}
Second, we obtains
\begin{align*}  
\sum_{\ell = 2}^{k-1} \sum_{j=2}^{\ell-2} \gamma^{k-2\ell+2j} (-1)^{\ell-j} f(F_\ell) f(F_j) &=
\sum_{r=2}^{k-3} \gamma^{k-2r} (-1)^{r} \sum_{j=2}^{k-r-1} f(F_{j+r}) f(F_j) \\
&= O\left( \sum_{r=2}^{k-3} \gamma^{k-2r} \sum_{j=2}^{k-1} f(F_j)^2 \right) \\
&= O(F_k).
\end{align*}
And third, we get
\begin{align*}
\sum_{\ell = 2}^{k-1} \sum_{j=2}^{\ell-2}  \gamma^{2\ell-k-2} (-1)^{k-\ell+1} f(F_\ell) f(F_j) &=
\sum_{\ell = 2}^{k-1} \gamma^{2\ell-k-2} (-1)^{k-\ell+1} f(F_\ell) \sum_{j=2}^{\ell-2} f(F_j) \\
&= \sum_{\ell = 2}^{k-1} \gamma^{2\ell-k-2}  |f(F_\ell)| \, O(1) \\
&= O\left(\left( \sum_{\ell = 2}^{k-1}  \gamma^{2k-4\ell}   \sum_{\ell = 2}^{k-1} f(F_\ell)^2 \right)^{1/2}\right) \\
&= O(F_k).
\end{align*}

Finally by Cauchy-Schwarz's inequality we obtain
\[
\sum_{n< F_L} |f(n)| \leqslant F_L^{1/2} \left( \sum_{n< F_L} f(n)^2 \right)^{1/2} = O\left( (F_L T_L)^{1/2} \right) = O(F_L) 
\]
as required. This completes the proof of the lemma.
\end{proof}

\vspace*{0.5cm}

\begin{lemma}\label{Ledifference2}
Suppose that $f(n)$ is real-valued and Z-additive. Then we have for all integers $1 \leqslant h\leqslant L/2$
\[
|\varphi_N(t) - \varphi_{F_L}(t)| \leqslant C_1\frac {\log N}{\gamma^{h}} + C_2  
\sum_{j=L-2h+1}^L  (1-\cos(t f(F_j))^{1/2}
\]
where $L = \lceil \log_\gamma (\sqrt 5 \, N )\rceil$ and $C_1$ and $C_2$ are two absolute positive constants.
\end{lemma}

\begin{proof}
For $r\geqslant  2$ we set 
\[
f_r(n) := \prod_{j=2}^{r-1} f(\delta_j(n) F_j)
\]  
and
\[
\varphi_N^{(r)}(t) := \frac 1N \sum_{n< N} e^{it f_r(n)}.
\]
The difference $\varphi_N(t) - \varphi_{F_L}(t)$ is now estimated in the 
following way:
\begin{align}
\left|\varphi_N(t) - \varphi_{F_L}(t) \right| &\leqslant 
\left| \varphi_N(t) - \varphi_{N}^{(L-h)}(t) \right| \nonumber \\
&+ \left|\varphi_{N}^{(L-h)}(t) -  \frac{F_{L-h-1}}{\gamma^{L-h-2}} \varphi_{F_{L-h-1}}(t)
- \frac{F_{L-h-2}}{\gamma^{L-h-1}} e^{it f(F_{L-h-1})}  \varphi_{F_{L-h-2}}(t) \right|  \label{eqdiffest} \\
&+ \frac{F_{L-h-1}}{\gamma^{L-h-2}}  \left|\varphi_{F_{L-h-1}}(t) - \varphi_{F_L}(t) \right| +
\frac{F_{L-h-1}}{\gamma^{L-h-2}} \left|\varphi_{F_{L-h-2}}(t) - \varphi_{F_L}(t) \right|.  \nonumber
\end{align}
Note that ${F_{L-h-1}}/{\gamma^{L-h-2}} + {F_{L-h-2}}/{\gamma^{L-h-1}} = 1$.

We consider now the four parts on the right hand side of (\ref{eqdiffest}) separately.
For the first part we observe that $f(n) = f_{L-h}(n)$ if $n < F_{L-h}$. If $n\geqslant  F_{L-h}$
then we have 
\begin{align*}
\left| e^{it f(n)} - e^{itf_{L-h}(n)} \right| &= \left| \prod_{j\geqslant  L-h} e^{it f(\delta_j(n) F_j)} - 1 \right| \\
&\leqslant \sum_{j= L-h}^L  \left| e^{it f(\delta_j(n) F_j)} - 1   \right|  \\
& = \sqrt 2 \sum_{j= L-h}^L  \sqrt{ 1- \cos( t f(\delta_j(n) F_j) ) } 
\end{align*}
and consequently
\[
\left| \varphi_N(t) - \varphi_{N}^{(L-h)}(t) \right| \leqslant \sqrt 2 \sum_{j= L-h}^L  \sqrt{ 1- \cos( t f(F_j) ) } .
\]

Next we use the fact (see \cite{S}) that for every $r\geqslant  2$ there exists a partition of $[0,1)$ into $F_r$ intervals $I_r(k)$, $0\leqslant k < F_r$,
of lengths $|I_r(k)| = \gamma^{2-r}$ if $0\leqslant k < F_{r-1}$ and $|I_r(k)| = \gamma^{1-r}$ if $F_{r-1}\leqslant k < F_r$ such that
\[
\{n \gamma\} \in I_r(k) \quad\mbox{if and only if} \quad (\delta_2(n), \delta_3(n),\ldots, \delta_{r-1}(n)) = (\delta_2(k), \delta_3(k),\ldots, \delta_{r-1}(k)).
\]
Since the discrepancy of the sequence $\{ n\gamma\}$ is of order $(\log N)/N$ we, thus, obtain
\begin{align*}
\varphi_{N}^{(r)}(t)  &= \sum_{k< F_r} \frac{ \#\{n< N : \{n \gamma\} \in I_r(k) \} }{N} e^{itf(k)}  \\
& = \sum_{k< F_r} \left( |I_r(k)| + O\left(\frac{\log N}{N} \right)\right) e^{itf(k)} \\
&= \sum_{k< F_r} |I_r(k)| e^{itf(k)}  + O\left(\frac{F_r\,\log N}{N}\right).
\end{align*}
Clearly we have
\[
\sum_{k< F_r} |I_r(k)| e^{itf(k)} = 
\frac{F_{r-1}}{\gamma^{r-2}} \varphi_{F_{r-1}}(t)
+ \frac{F_{r-2}}{\gamma^{r-1}} e^{it f(F_{r-1})}  \varphi_{F_{r-2}}(t)
\]
which implies (by setting $r = L-h$)
\[
\left|\varphi_{N}^{(L-h)}(t) -  \frac{F_{L-h-1}}{\gamma^{L-h-2}} \varphi_{F_{L-h-1}}(t)
- \frac{F_{L-h-2}}{\gamma^{L-h-1}} e^{it f(F_{L-h-1})}  \varphi_{F_{L-h-2}}(t) \right| \leqslant C \frac{\log N}{\gamma^h}
\]
for some constant $C> 0$.

For the final part we recall the defintion (\ref{eqAkdef}) of the matrices $A_k = A_k(t)$. We further set
\[
A := \left( \begin{array}{cc} 1 & 1 \\ 1 & 0 \end{array} \right)
\]
and note that $\| A - A_k\|_2 = |\eta_k|$ and $\| A_k\|_2 = \| A \|_2 = \gamma$. For exmaple, we have
\[
\| A^h - A_L A_{L-1} \cdots A_{L-h+1} \|_2 \leqslant \gamma^{h-1} \sum_{j=L-h+1}^L \| A - A_j \|_2 =  \gamma^{h-1}\sum_{j=L-h+1}^L  |\eta_j|.
\]
Note also that $A^h = \gamma^h M + O(\gamma^{-h})$ where $M$ is the matrix that projects to the direction $(\gamma, 1)$
and satisfies $M^2 = M$. By using the relation $F_L/F_{L-h} = \gamma^h + O(\gamma^{3h-2L})$ and the decomposition
\begin{align*}
A_L A_{L-1} \cdots A_2 - &\frac {F_L}{F_{L-h}} A_{L-h} A_{L-h-1}\cdots A_2 = 
A_L A_{L-1} \cdots A_2 - \gamma^h A_{L-h} A_{L-h-1}\cdots A_2 + O(\gamma^{2h-L}) \\
&= (A_L A_{L-1} \cdots A_{L-h+1} - A^h) A_{L-h} A_{L-h-1} \cdots A_2 \\
&+ (A^h - \gamma^h I)  A_{L-h} A_{L-h-1} \cdots A_2 + O(\gamma^{2h-L}) \\
\\
&= (A_L A_{L-1} \cdots A_{L-h+1} - A^h) A_{L-h} A_{L-h-1} \cdots A_2 \\
&+ \gamma^h (M -  I) A_{L-h} A_{L-h-1} \cdots A_2 + O(\gamma^{2h-L}) + O(\gamma^{L-2h})\\
\\
&= (A_L A_{L-1} \cdots A_{L-h+1} - A^h) A_{L-h} A_{L-h-1} \cdots A_2 \\
&+ \gamma^h (M -  I) (A_{L-h} A_{L-h-1} \cdots A_{L-2h+1} - A^h) A_{L-2h}\cdots  A_2  \\
&+ \gamma^h (M -  I) \gamma^h M  A_{L-2h}\cdots  A_2   + O(\gamma^{L-2h}),
\end{align*}
(where $I$ is the $2\times 2$ identity matrix) and by noting again that $(M-I)M = 0$, we, thus, obtain
\begin{align*}
\| A_L A_{L-1} \cdots A_2 - \frac {F_L}{F_{L-h}} A_{L-h} A_{L-h-1}\cdots A_2 \|_2 &\leqslant
\gamma^{L-2} \sum_{j=L-h+1}^L  |\eta_j| \\
&+ \gamma^{L-2} \sum_{j=L-2h+1}^{L-h}  |\eta_j| + O(\gamma^{L-2h}) 
\end{align*}
and consequently
\begin{align*}
|\varphi_{F_{L}}(t)- \varphi_{F_{L-h}}(t)| &\leqslant \sqrt 2 \; \frac{\gamma^{L-2}}{F_L} \sum_{j=L-2h+1}^L  |\eta_j| + O(\gamma^{-2h}) \\
&\ll \sum_{j=L-2h+1}^L  \sqrt{1-\cos(tf(F_j))} + \gamma^{-2h}.
\end{align*}
The same estimate holds if we replace $h$ by $h+1$ or $h+2$. 

Summing up, this completes the proof of the lemma.
\end{proof}

We are now ready to prove Theorem~\ref{ThZ2} that runs along the same lines as the proof of Theorems~\ref{Th2A} and \ref{Th2B},
in particular we use (again) the Berry-Esseen inequality. Instead of the upper bound (\ref{eqLe1}) we get
(with the help of the representation (\ref{eqvarphiproduct}) and the bound $\|A_j(t) A_{j+1}(t)\|_2 \leqslant \gamma^2 \exp(- c\,|\eta_j|^2)$)
the estimate
\[
|\varphi(t)| \leqslant \exp\left( - c_1' \, t^2 \sum_{ j\in  \tilde S(t)} f(F_j)^2 \right)
\]
which leads to an upper bound for $Q_F(1/T)$.\\

From Lemma~\ref{LeZZ} we get $\varphi(t) - \varphi_N(t) = O(|t|)$ which can be used to estimate the integral
$\int_{-1/T}^{1/T} \frac 1{|t|} |\varphi(t) - \varphi_N(t)| \, \d t$.\\ 

The integral $\int_{1/T \leqslant |t|\leqslant T} \frac 1{|t|} |\varphi(t) - \varphi_N(t)|\, \d t$ is (again) split into two
parts: 
\[
\int\limits_{1/T \leqslant |t|\leqslant T} \frac {|\varphi(t) - \varphi_{F_L}(t)|}{|t|}\, \d t + 
\int\limits_{1/T \leqslant |t|\leqslant T} \frac {|\varphi_{F_L}(t) - \varphi_{N}(t)|}{|t|}\, \d t. 
\]
For the second part we apply Lemma~\ref{Ledifference2}. Note that we choose $h\leqslant L/2$ in a way that
$\log T\,\log\,N / \gamma^h \ll 1/T$. 
For the first part we need proper bounds for the difference $\varphi(t) - \varphi_{F_L}(t)$.
For $t\to 0$ we certainly have
\begin{align*}
|\varphi(t) - \varphi_{F_L}(t)| &= | \varphi_{F_L}(t)| \left|  1 - \frac{\varphi(t)}{\varphi_{F_L}(t)} \right| \\
&\ll |t| \exp\left( - c_1' t^2 \sum_{j< L,\,  j\in  \tilde S(t)} f(F_j)^2 \right).
\end{align*}
By using the matrix product representation we also obtain (for all $t$)
\[
|\varphi(t) - \varphi_{F_L}(t)| \ll \exp\left( - c_1' \, t^2 \sum_{j< L,\,  j\in  \tilde S(t)} f(F_j)^2 \right).
\]
Finally by using the same methods as in the proof of Lemma~\ref{Ledifference2} we get
\begin{align*}
|\varphi(t) - \varphi_{F_L}(t)| &\ll \sum_{j\geqslant  L-h} |\eta_j| \cdot \gamma^{h-L} \left\| \prod_{j=2}^{L-h-1} A_j(t) \right\|_2 
+ O(\gamma^{-2h}) \\
& \ll |t| \sum_{j\geqslant  L-h} |f(F_j)| \exp\left( - c_1' \, t^2 \sum_{j< L-h,\,  j\in  \tilde S(t)} f(F_j)^2 \right) + O(\gamma^{-2h}) 
\end{align*}
Hence, we obtain
\begin{align*}
\int\limits_{1/T \leqslant |t|\leqslant T} \frac {|\varphi(t) - \varphi_{F_L}(t)|}{|t|}\, \d t  &\ll
\int_{1/T}^T \min\left\{ \frac 1{1+t}, \sum_{j\geqslant  L-h}  |f(F_j)| \right\} 
\exp\left( - c_2 t^2 \sum_{ j \in \tilde S(t),\, j <  L-h} f(F_j)^2 \right) \, \d t \\
&+ \frac{\log T}{\gamma^{2h}}.
\end{align*}
Note that $\log T / \gamma^{2h} \ll 1/T$ and so we are done.

\begin{ex}
Again we consider the most easy example (similarly to the $q$-ary case). We suppose that
\[
c_1 j^{-\alpha} \leqslant f(F_j) \leqslant c_2 j^{-\alpha}
\]
for $j\geqslant  2$ and some $\alpha>1$, where $c_1,c_2$ are positive constants. 
By applying Theorem~\ref{ThZ2}, choosing $h$ depending on $N$ as in Remark \ref{uprem} 
and by calculations that are very similiar to the $q$-ary case we obtain
\[
\| F - F_N\|_\infty \ll 
\left\{ \begin{array}{cl}
(\log N)^{1-\alpha} & \mbox{for $1< \alpha < 2$,} \\
\sqrt{\log\log N} (\log N)^{\frac \alpha 2} & \mbox{for $\alpha \ge 2$.}
\end{array}\right.
\]
\end{ex}

\begin{rem}
They are open questions how improve Theorem~\ref{ThZ2} and how far Theorems~\ref{ThZeckendorf} and \ref{ThZ2} can be generalized
to base sequences $G_n$ that satisfy linear recurrences with constant coefficients.
It should be certainly feasible to handle base sequences $(G_n)$ that are given
by $G_0=1$, $G_1=a$ and $G_{n+2}=a\,G_{n+1}+G_{n}$ ($n\geqslant 0$), where $a\geqslant 1$ is a given integer,
but even this case seems to be very involved.
\end{rem}

\begin{ack}
The authors are very grateful to Lukas Spiegelhofer and to G\'erald Tenenbaum for their valuable comments
to a previous version of this manuscript.
\end{ack}


\bibliographystyle{siam}
\bibliography{biblio}


\section*{Appendix A}

The following property is probably well known, however, we could not find a proper reference.
We thank Lukas Spiegelhofer (TU Wien) and G\'erald Tenenbaum (Nancy) for the following two nice proofs.

\begin{lemma}\label{LeAppendix}
Let be $\eta >0$ and let $(a_n)_n$ be a sequence of real numbers such that, for all $\tau \in (0,\eta]$
$$ \left\|\, \tau\,a_n \right\| \rightarrow 0,$$
where $\left\|\,\cdot\,\right\|$ is the distance to nearest integer function. Then $a_n \rightarrow 0$.
\end{lemma}

\begin{proof} [1st Proof (Spiegelhofer)]
By contradiction, we assume that $a_n$ does not converge to $0$. This means that there exist $\varepsilon >0$ and an infinite set $I$ of natural integers such that $|a_n|>\varepsilon$ for $n\in I$. Without loss of generality we can asssume that 
$\varepsilon \leqslant 1/\eta$. 

We define the following sets
$$ A_n := \left\{ \tau \in (0,\eta) : \left\| \tau\,a_n \right\| > \frac{\eta\,\varepsilon}{6} \right\}.$$
If $n\in I$ we observe that $\Lambda(A_n) \geqslant  \eta/3$, where $\Lambda$ denotes the Lebesgue measure.
We just have to remark that $a_n \cdot A_n$ results from $a_n \cdot(0,\eta)$ by cutting out $\leqslant  a_n \cdot \eta +1$ many intervals of length less or equal than $\eta\,\varepsilon /3 \leqslant  1/3$. Thus,
$$ \Lambda(A_n) \geqslant  \eta - \left(\eta+\frac{1}{a_n}\right) \frac{\eta\,\varepsilon}{3} \geqslant  \eta - \frac{2\eta}{3} = \frac{\eta}{3}.$$

Let us now consider the following sets
$$ 
B_m := \left\{ \tau \in (0,\eta] : \forall\,n \geqslant  m,\,\tau \notin A_n\right\}
= (0,\eta] \setminus  \bigcup_{n\geqslant  m} A_n.
$$
Note that $B_m \subseteq B_{m+1}$. Since $\left\|\, \tau\,a_n \right\| \rightarrow 0$ it follows 
that there exists $m=m(\tau)$ such that $\tau \in B_m$. Hence, $\bigcup_{m=0}^\infty B_m = (0,\eta]$.
Consequently the continuity of the Lebesgue measure implies
\begin{equation}\label{eqLambda}
\lim\limits_{m\rightarrow \infty} \Lambda(B_m) = \Lambda\left(\bigcup_{m=0}^{\infty} B_m\right) = \Lambda\left((0,\eta]\right) = \eta.
\end{equation}
However, for all $m\geqslant  0$, there exists $n\geqslant  m$ with $n\in I$. By definition $B_n \subseteq (0,\eta] \setminus A_n$, 
which implies $\Lambda\left(B_m\right) \leqslant \eta - \Lambda\left(A_n\right) \leqslant  2\eta/3$,
whence the contradiction to (\ref{eqLambda}).
\end{proof}

\begin{proof}[2nd Proof (Tenenbaum)]
By assumption $\| \tau\, a_n \| \to 0$ (for $\tau\in (0,\eta]$). Hence, the functions $f_n(\tau) = 1- \cos (2\pi \tau\, a_n)$ 
converges pointwise to $0$ on this interval. By Lebesgue's theorem it follows that, as $n\to\infty$,
\[
\frac 1{\eta} \int_0^\eta f_n(\tau)\, d\tau = 1 - \frac{\sin(2\pi \eta\, a_n)}{2\pi \eta\, a_n} \to 0.
\]
Thus, the sequence $(a_n)_n$ has to be bounded. Consequently, if $\tau > 0$ is sufficiently small we have 
$\| \tau\, a_n\| = |\tau\, a_n|$ which implies that that $|\tau\, a_n| \to 0$. Of course this also proves
$a_n\to 0$ as proposed.
\end{proof} 

\section*{Appendix B}

In Theorem~\ref{Th3} an explicit value for $c(\beta)$ is not given. However, 
G\'erald Tenenbaum mentioned to us that one can use, for example, the following explicit bound
\[
\| F - F_N \| \ll N^{- \overline{c}(\beta) }
(\log N)^{  \frac{\log(1/\beta)}{\log 2}},
\]
where
\[
\overline{c}(\beta) = \frac{\log(1/\beta) \log (2)}{\log(4/\beta)\log(2/\beta) + \log(2)^2}.
\]

\vspace*{0.5cm}

\begin{proof}[Proof (Tenenbaum)]
We first observe that $Q_F(1/T)$ can be upper bounded by the concentration function $Q_{G_M}(1/T)$ associated to the 
convolution of independent Bernoulli random variables $X_n$, $0\le n \le M$, 
with $\mathbb{P}[X_n = 0] = \mathbb{P}[X_n = \beta^n] = \frac 12$.
We choose $M = \lfloor \log(T)/\log(1/\beta) \rfloor$ so that $\beta^M \geqslant 1/T$.
Next let $R = 1 + \lfloor \log(4)/\log(1/\beta) \rfloor$, so that $\beta^R < 1/4$. 
Then $Q_{G_M}(1/T)$ is also the concentration function associated to the convolution of
the laws $F_a$, $0\le a < R$, where $F_a$ is the law of
\[
\sum_{0\le n\le M,\, n \equiv a \bmod R} X_n.
\]
However, for each $a$, the values of these random variables are ordered lexicographically and the
gaps are $> \beta^M \geqslant 1/T$. So we have $Q_{F_a}(1/T) \ll 2^{-M/R}$. Since the concentration of a convolution
product does not exceed that of the factors, we get
\[
Q_F(1/T) \le Q_{F_a}(1/T)  \ll 2^{- \log(T)/( \log(1/\beta) (1 + \log(4)/\log(1/\beta) ) ) } =   T^{-c_0(\beta)},
\]
where
\[
c_0(\beta) = \frac{\log 2 }{\log(4/\beta)}.
\]
\\
According to (\ref{eqTh2-ext-2}) we, thus,  have 
(with $L = \lfloor \log_2 N \rfloor$ and $h = \lfloor \log_2(T \log T) \rfloor$)
\[
\| F - F_N \| \ll T^{-c_0(\beta)} + T \beta^{L-h} \ll T^{-c_0(\beta)} + T^{1 + \frac{\log(1/\beta)}{\log 2} } 
N^{- \frac{\log(1/\beta)}{\log 2} } (\log T)^{ \frac{\log(1/\beta)}{\log 2}} .
\]
Hence by choosing 
\[
T = N^{ \frac{ \log(1/\beta)/\log 2 }{ 1+ \log(1/\beta)/\log 2 + c_0(\beta) } }
\]
we finally obtain
\[
\| F - F_N \| \ll N^{- \frac{  c_0(\beta) \log(1/\beta)/\log 2   }{ 1+ \log(1/\beta)/\log 2 + c_0(\beta) }   }
(\log N)^{  \frac{ \log(1/\beta)}{\log 2}} 
= N^{-\overline{c}(\beta)} (\log N)^{\frac{ \log(1/\beta)}{\log 2}} 
\]
as proposed.

\end{proof}

\end{document}